\documentclass[12pt]{amsart}
\usepackage{amssymb}
\usepackage{amsmath,amscd}
\usepackage{amsfonts,latexsym,verbatim,amscd,mathrsfs,color,array}
\usepackage[all,cmtip]{xy}
\usepackage{mathrsfs}
\usepackage{multirow}
\usepackage{graphicx}
\usepackage{amsfonts, amsthm}
\usepackage{bbm}
\usepackage[hmargin=1in,vmargin=1in]{geometry}
\usepackage{mathdots}
\usepackage{stmaryrd}
\usepackage{MnSymbol}
\usepackage{bbm}
\usepackage{enumitem}
\usepackage[hidelinks]{hyperref}

\allowdisplaybreaks

\def\XXint#1#2#3{{\setbox0=\hbox{$#1{#2#3}{\int}$ }
		\vcenter{\hbox{$#2#3$ }}\kern-.6\wd0}}

\DeclareMathSymbol{\intprod}{\mathbin}{MnSyC}{'270}
\newcommand{\LB}{\left[}
\newcommand{\RB}{\right]}
\newcommand{\LA}{\langle}
\newcommand{\RA}{\rangle}

\newcommand{\N}{{\mathbb N}}

\newcommand{\R}{{\mathbb R}}

\newcommand{\inj}{{\mathrm{inj} }}

\newcommand{\del}{{\partial}}

\newtheorem{thm}{Theorem}[section]
\newtheorem{lemma}[thm]{Lemma}
\newtheorem*{lemma*}{Lemma}
\newtheorem{prop}[thm]{Proposition}

\newtheorem{cor}[thm]{Corollary}
\newtheorem{conj}[thm]{Conjecture}
\newtheorem*{conj*}{Conjecture}

\newtheoremstyle{others}
{3pt}
{2pt}
{}
{}
{\bf}
{.}
{.5em}
{}

\theoremstyle{others}
\newtheorem{rmk}[thm]{Remark}
\newtheorem*{rmk*}{Remark}

\numberwithin{equation}{section}

\begin{document}

\title{Uhlenbeck compactness for Yang-Mills flow in higher dimensions}
\author{Alex Waldron}
\address{University of Wisconsin, Madison}
\email{waldron@math.wisc.edu}

\begin{abstract}
This paper proves a general Uhlenbeck compactness theorem for sequences of solutions of Yang-Mills flow on Riemannian manifolds of dimension $n \geq 4,$ including rectifiability of the singular set at finite or infinite time.
\end{abstract}

\maketitle

\tableofcontents

\thispagestyle{empty}

\section{Introduction}

\subsection{Background}

This article generalizes the well-known sequential compactness theorems for Yang-Mills connections, due to Uhlenbeck \cite{uhlenbecklp} and Nakajima \cite{nakajimacompactness}, to solutions of the Yang-Mills flow
\begin{equation}\tag{YM}
\frac{\partial A}{\partial t} = - D_A^*F_A.
\end{equation}
Here $A(t)$ is a time-dependent family of connections on a vector bundle $E$ over a Riemannian manifold.
For background on Yang-Mills flow (YM), we refer the reader to the textbook of Donaldson and Kronheimer \cite{donkron}, \S 6, or \cite{instantons}, \S 2. Detailed expository treatments of Uhlenbeck's compactness theory for gauge fields (connections) have appeared in \cite{donkron}, \S 2 and 4, and the textbook of Wehrheim \cite{wehrheimuhlenbeck}.


Several notable compactness theorems have been established for geometric flows, building on prior work for the corresponding elliptic equation. 
The first is due to Brakke \cite{brakkeflow} in the case of mean curvature flow, generalizing Allard's compactness theorem \cite{allardfirstvariation} for minimal submanifolds with locally bounded area. 
Hamilton \cite{hamiltoncompactness} proved a smooth compactness theorem for solutions of the Ricci flow with uniformly bounded Riemann tensor and positive injectivity radius, which has been substantially extended by Chen and Wang \cite{chenwangspaceofricciI, chenwangspaceofricciII} (see also Bamler \cite{bamlerstructuretheory, bamlerconvergenceannals}). These are parabolic analogues of the celebrated compactness theory for manifolds with controlled Ricci tensor, due to Cheeger \cite{cheegerfiniteness}, Gromov \cite{gromovstructures}, Anderson \cite{andersonricci}, Cheeger-Colding-Tian \cite{cheegercoldingtian}, and others. 
In the case of harmonic maps and harmonic map flow, respectively, the most general compactness theorems are due to Lin \cite{linharmonicannals} and Lin-Wang \cite{linwangharmonic} (see also the textbook \cite{linwangbook}). 

Compactness results of this kind are primarily used for analyzing individual solutions of the flow---meaning, the structure of finite-time singular sets, infinite-time convergence, and (the ubiquitous) blowup arguments. In the context of Yang-Mills flow on compact K\"ahler manifolds \cite{sibleyasymptotics, sibleywentworth}, one typically relies on the work of Hong and Tian \cite{hongtian}. Our main motivation is to strengthen the results of \cite{hongtian} in order to apply them outside the K\"ahler context. Specifically, our analysis includes the case of finite-time blowup as well as infinite-time blowup in the absence of a holomorphic structure. 
These results have been used to obtain a partial characterization of the infinite-time singular set for Yang-Mills flow on manifolds of special holonomy \cite{oliveirawaldron}.

As far as the analysis is concerned, the case of (YM) differs from those discussed above in that certain aspects are simpler---for instance, we are able to prove rectifiability of the singular set using a Laplace-transform trick combined with Preiss's Theorem---while other aspects are more difficult, especially those related to the gauge freedom of the underlying bundle. 
%


\subsection{Statement of results}

%


\begin{thm}\label{thm:sigma}
Let $M$ be a Riemannian manifold (without boundary) of dimension $n \geq 4.$ Fix $0 < \tau < \infty,$ and let $\{ A_i(x,t) \}$ be a sequence of smooth solutions of (YM) on $M \times \LB 0 , \tau \right).$ Writing $F_i(t) = F_{A_i(t)},$ 
assume that for any compactly contained open set $U \Subset M,$ there holds
\begin{equation}\label{sigma:finitenergy}
\sup_{\stackrel{i\in \N}{0 \leq t < \tau}} \int_{U} |F_i (t)|^2 \, dV < \infty.
\end{equation}
For $\epsilon_0 > 0$ sufficiently small (depending only on $n$), define
\begin{equation}\label{sigma:sigmadef}
\Sigma = \lbrace x \in M \ | \ \liminf_{R \searrow 0} \liminf_{i \to \infty} \Phi\left(A_i; R, x, \tau - R^2 \right) \geq \epsilon_0 \rbrace
\end{equation}
where $\Phi$ is the weighted energy functional given by (\ref{phidef}) below. Then, $\Sigma$ is closed and of locally finite $(n-4)$-Hausdorff measure.


Let $\tau_i \nearrow \tau.$ Choose any subsequence, again denoted by $\{i\},$ such that the limit of measures
\begin{equation*}
\mu = \lim_{i \to \infty} \left( |F_i(x, \tau_i)|^2 \, dV \right)
\end{equation*}
exists, and redefine 
$\Sigma$ according to (\ref{sigma:sigmadef}). If, for each $U \Subset M,$ there holds
\begin{equation}\label{sigma:dstarfass}
\lim_{\sigma \nearrow \tau} \liminf_{i \to \infty} \int_{\sigma}^\tau \!\!\!\!\! \int_U |D^*F_i|^2 \, dV dt = 0
\end{equation}
then $\Sigma$ is $(n-4)$-rectifiable, \textit{i.e.}, consists of a countable union of $(n-4)$-dimensional Lipschitz submanifolds (up to $\mathcal{H}^{n-4}$-measure zero).
\end{thm}

\begin{rmk} Afuni \cite{afuni} has independently obtained the finiteness of $\mathcal{H}^{n-4}(\Sigma).$ 
\end{rmk}


\begin{thm}\label{thm:uhlenbeck}
Let $\{A_i\}$ be as in Theorem \ref{thm:sigma}, satisfying (\ref{sigma:finitenergy}). 
Passing to a further subsequence, 
there exists 
a smooth connection $A_\infty$ on a bundle $E_\infty \to M \setminus \Sigma,$ and an exhaustion of open sets
$$U_1 \Subset \cdots \Subset U_i \Subset U_{i+1} \Subset \cdots \subset M \setminus \Sigma, \qquad \bigcup_{i = 1}^{\infty} U_i = M \setminus \Sigma,$$
together with bundle maps $u_i: \left. E \right|_{U_i} \to \left. E_\infty \right|_{U_i}$
(independent of time), as follows. For any sequence of times $t_i \nearrow \tau,$ we have
\begin{equation}\label{uhlenbeck:titotau}
u_i \left( A_i(t_i) \right)  \to A_\infty  \mbox{ in } C_{loc}^\infty(M \setminus \Sigma).
\end{equation}
Moreover, if
\begin{equation}\label{dstarftozero}
\int_0^\tau \!\!\!\!\! \int_{U_k} |D^*F_i |^2 \, dV dt \to 0 \mbox{ as } i \to \infty
\end{equation}
for each $k,$ then $A_\infty$ is a Yang-Mills connection, and 
\begin{equation}\label{uhlenbeck:uiaiconv}
u_i \left(A_i \right)  \to \underline{A}_\infty  \mbox{ in } C_{loc}^\infty \left( \left( M \setminus \Sigma \right) \times (0, \tau) \right),
\end{equation}
where $\underline{A}_\infty$ is the constant solution of (YM) equal to $A_\infty.$
\end{thm}

 

\begin{cor}[{\it Cf.} Hong-Tian \cite{hongtian}, Theorem A]\label{cor:uhlenbeck} Fix an arbitrary sequence $t_i \nearrow \infty,$ and let $A(t)$ be a smooth solution of (YM) on $M \times \LB 0 , \infty \right),$ with $M$ compact. 
After passing to a subsequence of $t_i,$ there exists a closed, $(n-4)$-rectifiable set $\Sigma \subset M,$ together with an Uhlenbeck limit $A_\infty,$ which is a smooth Yang-Mills connection on $E_\infty \to M \setminus \Sigma,$ such that
\begin{equation*}
u_i \left( A(t_i + t) \right)  \to \underline{A}_\infty  \mbox{ in } C_{loc}^\infty \left( \left( M \setminus \Sigma \right) \times \R \right).
\end{equation*}
\end{cor}



\begin{rmk}
Rescaling arguments at the blowup set of (YM) have been carried out by several authors, including Schlatter \cite{schlatterlongtime} in dimension four, and Weinkove \cite{weinkovesingularityformation} for Type-I singularities in higher dimensions.
The next theorem follows from the refined blowup analysis due to Lin \cite{linharmonicannals}, Lin-Wang \cite{linwangharmonic}, and Tian \cite{tiancalibrated}, applied to (YM) in the K\"ahler context by Hong and Tian \cite{hongtian}. 
The recent paper by Kelleher and Streets \cite{kelleherstreets, kelleherstreetscorrigendum}, \S 6, also contains an essentially correct proof of the following result. 
\end{rmk}

\begin{thm}[\textit{Cf.} Kelleher-Streets \cite{kelleherstreets}, \S 6, Hong-Tian \cite{hongtian}, Theorem C, Lin-Wang \cite{linwangbook}, \S 8.5, Tian \cite{tiancalibrated}, \S 3-4]\label{thm:linwangblowup} Let $A_i$ be as in Theorem \ref{thm:sigma}, satisfying (\ref{sigma:finitenergy}) and (\ref{sigma:dstarfass}). 
For $\mathcal{H}^{n-4}$-a.e. $x_0 \in \Sigma,$ we may pass to a subsequence for which there exist $x_i \to x_0,$ $t_i \nearrow \tau,$ $\lambda_i \searrow 0,$ 
and gauge transformations $u_i,$ such that 
\begin{equation}
\lambda_i u_i \left( A_i \right) \left(x_i + \lambda_i x, t_i + \lambda_i^2 t  \right) \to \underline{B}(x) \mbox{ in } C^\infty_{loc} \left( \R^{n + 1} \right).
\end{equation}
Here, $\underline{B}$ is a constant solution of (YM) on $T_{x_0} M \cong \R^n,$
which is the product of a flat connection on $T_{x_0} \Sigma$ with a nontrivial finite-energy
Yang-Mills connection on $\left( T_{x_0} \Sigma\right)^{\perp} \cong \R^4.$
\end{thm}

\begin{rmk}
In the case that $T < \infty$ and $A_i(t) = A(t)$ is a single smooth solution of (YM) over $M \times \LB 0 , T \right),$ the assumption (\ref{sigma:dstarfass}) is automatically satisfied for $\tau = T,$ and we conclude from Theorem \ref{thm:sigma} that the finite-time singular set is rectifiable. The proof (based on Lemma \ref{lemma:laplacelemma}) also works for the case of harmonic map flow, where this result has not appeared in the literature. 
Based on \cite{lte}, 
however, it is natural to conjecture that for (YM), the result is vacuously true. 
\end{rmk}

\begin{conj} Let $A(t)$ be a smooth solution of (YM) on $M \times \LB 0, T \right),$ with $T < \infty,$ satisfying
\begin{equation*}
\sup_{0 \leq t < T} \int_M |F(t) |^2 \, dV  < \infty.
\end{equation*}
Then, at time $\tau = T,$ $\mathcal{H}^{n-4}(\Sigma) = 0$ and the defect measure (see (\ref{sigma:nudef}) below) vanishes identically. 
\end{conj}

\section{Technical results}

\subsection{Hamilton's monotonicity formula} We recall the basic monotonicity formula for (YM) in higher dimensions, due to Hamilton \cite{hamiltonmonotonicity}.

For $x, y \in M$ and $R > 0,$ let
\begin{equation}\label{udef}
u_{R,x}(y) = \frac{R^{4-n}}{(4\pi)^{n/2}} \exp -\left(\frac{d \left( x, y \right)}{2R} \right)^2.
\end{equation}
Fix a smooth cutoff function $\varphi(r)$ supported on the unit interval, with $\varphi(r) \equiv 1$ on $\LB 0, 1/2 \RB,$ and let 
\begin{equation*}
\varphi_{x, \rho} (y) = \varphi \left( \frac{d(x,y)}{ \rho} \right)
\end{equation*}
for $\rho > 0.$
Also let
\begin{equation*}
\rho_1(x) = \min \LB \frac{\inj(M,x)}{2}, \frac{\sqrt{\tau}}{2}, 1 \RB
\end{equation*}
and
\begin{equation*}
\varphi_x(y) = \varphi_{x,\rho_1(x)} (y).
\end{equation*}
Given a solution $A = A(t)$ of (YM), define
\begin{equation}\label{phidef}
\Phi(A; R,x,t) =  \int_M |F_{A(t)}(y)|^2 u_{R,x}(y) \varphi^2_{x}(y) \, dV_y \qquad 
\end{equation}
\begin{equation}\label{xidef}
\qquad \Xi(A; R, x, t_1, t_2) = \int_{t_1}^{t_2} \!\!\!\!\! \int_M |D_{A(t)}^*F_{A(t)}(y)|^2 u_{R,x}(y) \varphi^2_{x}(y) \, dV_y dt.
\end{equation}
We shall typically suppress $A$ and write $F(t) = F_{A(t)},$ $D^*F(t) = D_{A(t)}^* F_{A(t)},$ etc.

Let $U \Subset U_1 \Subset M$ be compactly contained open submanifolds such that 
\begin{equation}\label{brho1}
B_{\rho_1}(x) \subset U_1
\end{equation}
for all $x \in U.$ In keeping with (\ref{sigma:finitenergy}), we shall assume
\begin{equation}\label{efiniteenergy}
\sup_{0 \leq t < \tau} \int_{U_1} \left| F(t) \right|^2 \, dV \leq E
\end{equation}
for a constant $E > 0.$

\begin{thm}[Hamilton \cite{hamiltonmonotonicity}]\label{thm:hamilton}
For $R_0 \geq R_1 \geq R_2 > 0,$
$R_1^2 \leq t < \tau,$ and any $x \in U,$
we have
\begin{equation}\label{hamilton:mainestimate}
\begin{split}
\Phi(R_2, x, t) & \leq e^{C_0  (R_1 - R_2) } \Phi(R_1, x, t - R_1^2 + R_2^2) + C_1 \left( R_1^2 - R_2^2 \right) E.
\end{split}
\end{equation}
Here $C_0$ and $C_1$ depend on the geometry of $M$ near $U.$ In particular, for any $\epsilon > 0,$ taking $R_0$ sufficiently small, we have
\begin{equation}\label{hamilton:epsilonestimate}
\Phi(R_2, x, t) \leq \left( 1 + \epsilon \right) \Phi(R_1, x, t - R_1^2 + R_2^2) + \epsilon.
\end{equation}
\end{thm}
\begin{proof} This version of Hamilton's formula corresponds to the case $\gamma =1$ in Theorem 5.7 
of \cite{oliveirawaldron}.
\end{proof}

\subsection{Basic estimates for Yang-Mills flow} This section adapts to general dimension several basic results for Yang-Mills flow in dimension four, from \cite{instantons}, \S 3 and \cite{lte}, \S 3. These estimates have not appeared in the literature in this level of detail.

Let $U \Subset U_1 \Subset M$ satisfy (\ref{brho1}) as above.

\begin{lemma}\label{lemma:basicepsreg} 
Let $E, E_0 > 0,$ and $k \in \N.$ There exists a constant $\epsilon_0 > 0,$ depending on $E_0$ and $n,$\footnote{
We find this formulation of Lemma \ref{lemma:basicepsreg} to be more natural than the corresponding version in which $\epsilon_0$ depends only on $n$ (see Theorem 6.4 
of \cite{oliveirawaldron}). In the proof of Theorems \ref{thm:sigma}-\ref{thm:uhlenbeck}, the dependence on $E_0$ can easily be removed.} as well as $R_0 >0,$ depending on $E, k,$ and the geometry of $M$ near $U,$ as follows.

Let $A(t)$ be solution of (YM) on $M \times \LB 0, T \right),$ satisfying (\ref{efiniteenergy}). Assume that for some $0 < R < R_0,$  $x \in U, $ and $R^2 \leq t_0 \leq T,$ there hold
\begin{equation}\label{basicepsreg:E0e0assn}
\Phi \left( 2R, x, t_0 - R^2 \right) \leq E_0, \qquad
\Phi \left( R, x, t_0 - R^2 \right) < \epsilon_0.
\end{equation}
Then, for $t_0 - \frac{R^2}{2} \leq t < t_0,$ we have
\begin{equation}\label{basicepsreg:4}
\| F(t)\|_{C^0 \left( B_{R/2} \right)} \leq \frac{C_n}{R^2}.
\end{equation}
Letting
\begin{equation}\label{basicepsreg:fdstarfassn}
\begin{split}
\sup_{t_0 - R^2 \leq t < t_0} R^{4 - n} \int_{B_R(x)} |F(t)|^2 \, dV = \epsilon, \qquad
R^{4-n} \int_{t_0 - R^2}^{t_0} \!\! \int_{B_R(x)} \left| D^*F \right|^2 \, dV dt = \delta^2 
\end{split}
\end{equation}
then for $t_0 - \frac{R^2}{4} \leq t < t_0,$ we also have
\begin{equation}\label{basicepsreg:derivests}
\begin{split}
\| \nabla^{(k)} F(t) \|_{C^0\left( B_{R / 4}(x) \right)} \leq C_{k,n} R^{-2-k} \sqrt{\epsilon} , \qquad
\|\nabla^{(k)} D^*F(t)\|_{C^0 \left( B_{R/4}(x) \right)} \leq C_{k,n} R^{-3-k} \, \delta.
\end{split}
\end{equation}
\end{lemma}
\begin{proof} The estimate (\ref{basicepsreg:4}) follows from (\ref{basicepsreg:E0e0assn}) by the $\epsilon$-regularity theorem (see \cite{chenshenmonotonicity}, \cite{struwehm}, or Theorem 6.2 
 of \cite{oliveirawaldron} for this version).
Upon rescaling $B_{R/2}$ to a unit ball $\tilde{B}_1$ and letting $t_0 = 0,$ (\ref{basicepsreg:4}) becomes the uniform bound 
\begin{equation*}
\sup_{\tilde{B}_1 \times \LB -1, 0 \right)} |F(x,t)| \leq C.
\end{equation*}
The assumptions (\ref{basicepsreg:fdstarfassn}) become
\begin{equation*}
\sup_{-1 \leq t < 0} \| F (t) \|^2_{L^2(\tilde{B}_1)} = \epsilon, \qquad \int_{-1}^0 \!\! \int_{\tilde{B}_1} \left| D^*F \right|^2 \, dV dt = \delta^2.
\end{equation*}
The scale-invariant estimates (\ref{basicepsreg:derivests}) now follow from the standard Moser iteration and bootsrapping argument of \cite{instantons}, Proposition 3.2.
\end{proof}

\begin{lemma}\label{lemma:antibubble} 
For $0 < R < R_0,$  $0 \leq t_1 \leq t_2 < T,$ and $x \in U,$ assume 
\begin{equation}
\begin{split}
\sup_{t_1 \leq t \leq t_2} \Phi \left( 2R, x, t \right) \leq E_0, \quad \quad
\Xi \left(R, x , t_1, t_2 \right) \leq \xi^2 
\end{split}
\end{equation}
and put
$$\gamma = \xi \left(\xi + \frac{\sqrt{\left( t_2 - t_1 \right) E_0 }}{R} \right).$$
Then
\begin{equation}\label{antibubble:aest}
\left| \Phi \left(R,x, t_2 \right) - \Phi \left(R, x , t_1 \right) \right| \leq C_n \gamma.
\end{equation}
For $0 < \epsilon < \epsilon_0,$ if
\begin{equation}\label{antibubble:best}
\Phi \left(R, x, t_2 \right) + \gamma \leq \epsilon \quad \mbox{or} \quad \Phi \left(R, x, t_1 \right) + \gamma \leq \epsilon
\end{equation}
then, for $t_1 + \frac{3}{4}R^2 \leq t \leq t_2,$ 
there hold
\begin{equation}\label{antibubble:derivests}
\begin{split}
\| \nabla^{(k)} F(t) \|_{C^0\left( B_{R/4}(x) \right)} \leq C_{k,n} R^{-2-k} \sqrt{\epsilon} , \qquad
\|\nabla^{(k)} D^*F(t)\|_{C^0 \left( B_{R/4}(x) \right)} \leq C_{k,n} R^{-3-k} \, \xi.
\end{split}
\end{equation}
\end{lemma}
\begin{proof} Integrating by parts once against $u_{R,x} \varphi_{x}^2$ in the pointwise energy identity (2.4) 
of \cite{lte}, and integrating in time, we obtain
\begin{equation}\label{antibubble:1}
\begin{split}
\Phi(R, x,t_2) - \Phi(R, x, t_1) & = -2 \Xi(R, x, t_1, t_2)  \\
& - 2 \int_{t_1}^{t_2} \!\!\!\! \int \LA D^*F^i, F_{ij} \RA \left( \nabla^j u_{R,x} \varphi^2_{x} + 2 u_{R,x} \varphi_{x} \nabla^j \varphi_{x} \right) \, dV dt.
\end{split}
\end{equation}
Letting $r = d(x, y),$ we have
\begin{equation}\label{uxr2calc}
\begin{split}
\left| \nabla u_{R,x} \right| = \frac{R^{4-n}}{(4\pi)^{n/2}} \frac{r}{2R^2} \, \exp -\left(\frac{ r}{2R} \right)^2 & = \frac{1}{R} \sqrt{u_{R,x}} \left( \frac{R^{4-n}}{(4\pi)^{n/2}} \right)^{1/2} \frac{r}{2R} \, \exp -\frac{ r^2}{8R^2} \\
& \leq \frac{C}{R} \sqrt{u_{R,x}} \sqrt{u_{2R,x}} \frac{r}{R} \, \exp \left( -\frac{ r^2}{16R^2} \right) \\
& \leq \frac{C}{R} \sqrt{u_{R,x}} \sqrt{u_{2R,x}}.
\end{split}
\end{equation}
We may therefore apply H\"older's inequality to estimate
\begin{equation}\label{antibubble:2}
\begin{split}
\left| \int_{t_1}^{t_2} \!\!\!\!\! \int \LA D^*F^i, F_{ij} \RA \nabla^j u_{R,x} \varphi^2_{x}  \, dV dt \right| & \leq \frac{C}{R} \sqrt{\Xi(R, x, t_1, t_2)} \sqrt{\int_{t_1}^{t_2} \Phi(2R, x, t) \, dt} \\
& \leq C \frac{\xi \sqrt{\left( t_2 - t_1 \right) E_0}}{R}.
\end{split}
\end{equation}
Next, we estimate 
\begin{equation}\label{antibubble:3}
\begin{split}
& \left| \int_{t_1}^{t_2} \!\!\!\!\! \int \LA D^*F^i, F_{ij} \RA u_{R,x} \varphi_{x} \nabla^j \varphi_{x}  \, dV dt \right| \\
& \qquad\qquad\qquad\qquad \qquad\qquad \leq C \sqrt{\Xi(R, x, t_1, t_2)} \sqrt{ \int_{t_1}^{t_2} \!\!\!\!\! \int \left| F \right|^2 u_{R,x} \left| \nabla \varphi_{x} \right|^2 \, dV dt} \\
&\qquad\qquad\qquad\qquad \qquad\qquad \leq C \xi e^{-\frac{\rho_1(x)^2}{4R^2}}\rho_1(x)^{-1} \sqrt{\left( t_2 - t_1 \right) E_0}.
\end{split}
\end{equation}
We may assume that $R_0$ 
is sufficiently small that
$$e^{ -\frac{\rho_1(x)^2}{4R_0^2} } \leq \rho_1(x)$$
for all $x \in U.$ Then, inserting (\ref{antibubble:2}-\ref{antibubble:3}) into (\ref{antibubble:1}) yields (\ref{antibubble:aest}).

Under the assumption (\ref{antibubble:best}), we conclude from (\ref{antibubble:aest}) that
\begin{equation*}
\Phi(R,x, t) \leq C \epsilon
\end{equation*}
for $t_1 \leq t \leq t_2.$ The desired bounds (\ref{antibubble:derivests}) now follow from Lemma \ref{lemma:basicepsreg}.
\end{proof}

\begin{lemma}\label{lemma:distancecomp} Fix $0 < \tau_0 < \tau \leq T.$ Let $K, \delta > 0$ and $\tau_0 \leq t_i < \tau,$ for $i = 1,2,$ be such that 
\begin{equation}\label{distancecomp:deltatiassn}
\delta^2 | t_2 - t_1 | \leq K^2.
\end{equation}
Assume that
\begin{equation}
\sup_{\stackrel{x \in U_1}{0 \leq t < \tau}}\left| F(x,t) \right| \leq K
\end{equation}
and
\begin{equation}
\int_0^\tau \!\!\!\!\! \int_{U_1} |D^*F|^2 \, dV dt \leq \delta^2.
\end{equation}
Then 
\begin{equation}\label{distancecomp:linfty}
\|A(t_2) - A(t_1)\|_{C^0(U)} \leq C_{\ref{lemma:distancecomp}} \delta \sqrt{|t_2 - t_1|}.
\end{equation}
Fixing a reference connection $\nabla_{ref}$ on $E$ and defining the $C^k$ norms accordingly, for $k \in \N,$ we have
\begin{equation}\label{distancecomp:kbound}
\begin{split}
\|A(t_2) - A(t_1)\|_{C^k( U )} & \leq C_{\ref{lemma:distancecomp}} \delta \sqrt{| t_2 - t_1| }  \left(1 + \|A(t_1) \|^{k!}_{C^{k-1}\left( U\right) } \right).
\end{split}
\end{equation}
The constants $C_{\ref{lemma:distancecomp}}$ depend on $K, k, \tau_0,\nabla_{ref},$ and the geometry of $M$ near $U.$ 
\end{lemma}
\begin{proof} In this proof, the constant $C$ will have the dependence of $C_{\ref{lemma:distancecomp}}.$ We shall assume $t_1 \leq t_2,$ since the opposite case follows by a similar argument.

First note that by covering $U$ with finitely many balls and applying Lemma \ref{lemma:basicepsreg}, for any $\tau_0 \leq t < \tau,$ we may obtain an estimate
\begin{equation}\label{distancecomp:dstarfbound}
\|D^*F(t) \|_{C^0( U )} \leq C \|D^*F\|_{L^2 \left( U_1 \times \LB t - \tau_0, t \RB \right)}.
\end{equation}
To prove (\ref{distancecomp:linfty}), using (\ref{distancecomp:dstarfbound}) and H\"older's inequality, we calculate
\begin{equation}\label{omega1}
\begin{split}
\|A(t_2) - A(t_1)\|_{C^0 \left( U \right)} & \leq \int_{t_1}^{t_2} \|D^*F(t)\|_{C^0( U )} \, dt \\
& \leq C \int_{t_1}^{t_2} \|D^*F\|_{L^2 \left( U_1 \times \LB t - \tau_0, t \RB \right)} \, dt \\
& \leq C (t_2- t_1)^{1/2} \left( \int_{t_1}^{t_2} \|D^*F\|^2_{L^2\left( U_1 \times \LB t-\tau_0, t \RB \right)} \, dt \right)^{1/2} \\
& \leq C (t_2 - t_1)^{1/2} \left( \int_{t_1}^{t_2} \!\!\!\! \int_{t - \tau_0}^t \|D^*F(s) \|_{L^2\left( U_1 \right)}^2 \, ds dt \right)^{1/2}.
\end{split}
\end{equation}
The domain of integration
$$t - \tau_0 \leq s \leq t, \quad \quad t_1 \leq t \leq t_2$$
may be relaxed to
$$0 \leq s \leq \tau, \quad \quad s \leq t \leq s + \tau_0.$$ Then (\ref{omega1}) becomes
\begin{equation*}
\begin{split}
\|A(t_2) - A(t_1)\|_{C^0 \left( U \right)} & \leq C (t_2 - t_1)^{1/2} \tau_0^{1/2} \left( \int_{0}^{\tau} \|D^*F(s)\|_{L^2\left( U_1 \right)}^2 \, ds \right)^{1/2} \\
& \leq C (t_2 - t_1)^{1/2} \delta
\end{split}
\end{equation*}
which is (\ref{distancecomp:linfty}).

Next, we calculate as follows:
\begin{equation*}\label{derivcomparison}
\begin{split}
\partial_t \nabla_{ref} A & = - \nabla_{ref} D^*F \\
& = - \nabla_A D^*F + A \# D^*F, \\
\partial_t \nabla_{ref}^{(2)} A & = - \nabla_{ref}^{(2)} D^*F \\
& = - \nabla_A^{(2)} D^*F + A \# \nabla_A D^*F + \nabla_{ref}A \# D^*F + A \# A \# D^*F, \mbox{ etc.}
\end{split}
\end{equation*}
Continuing in this fashion, we obtain bounds
\begin{equation}\label{distancecomp:dtkbound}
\| \partial_t \nabla_{ref}^{(k)} A (t) \|_{C^0( U )} \leq C \left( 1 + \|A(t) \|^k_{C^{k-1} \left( U \right)} \right) \sum_{\ell = 0 }^{k} \| \nabla^{(\ell)} D^*F (t) \|_{C^0 \left( U \right)}
\end{equation}
for each $k \in \N.$ Integrating (\ref{distancecomp:dtkbound}) in time, and applying Lemma \ref{lemma:antibubble} and H\"older's inequality as above, we obtain
\begin{equation}\label{distancecomp:protobound}
\begin{split}
\|A(t_2) - A(t_1)\|_{C^k( U )} & \leq C \delta \sqrt{ t_2 - t_1 } \left(1 + \sup_{t_1 \leq t \leq t_2 } \| A(t) \|^{k}_{C^{k-1}\left( U \right) } \right)
\end{split}
\end{equation}
for $k \in \N.$

To obtain (\ref{distancecomp:kbound}) from (\ref{distancecomp:protobound}), we use induction. The base case $k = 0$ is (\ref{distancecomp:linfty}). Assuming that (\ref{distancecomp:kbound}) holds for $k - 1,$ for any $t_1 \leq t \leq t_2,$ we have 
\begin{equation}\label{distancecomp:induction}
\begin{split}
\| A(t) \|^{k}_{C^{k-1}\left( U \right)} & \leq \left( \|A(t_1)\|_{C^{k-1}\left( U \right)} + \|A(t) - A(t_1)\|_{C^{k-1}\left( U \right)} \right)^{k} \\
& \leq \left( \|A(t_1)\|_{C^{k-1}\left( U \right)} + C K \left( 1 + \|A(t_1)\|^{(k - 1)!}_{C^{k - 2}\left( U \right)} \right) \right)^{k}\\
& \leq C \left( 1 + \|A( t_1 )\|_{C^{k-1}\left( U \right)}^{k!} \right).
\end{split}
\end{equation}
We have used the induction hypothesis and (\ref{distancecomp:deltatiassn}) in the second line. Substituting (\ref{distancecomp:induction}) into (\ref{distancecomp:protobound}) gives (\ref{distancecomp:kbound}) for $k,$ completing the induction.
\end{proof}

\subsection{Weighted density for Preiss's Theorem}
This section contains a lemma that will allow us to appeal directly to Preiss's Rectifiability Theorem in the parabolic context.

\begin{lemma}\label{lemma:laplacelemma} Let $\mu$ be a locally finite measure on $M^n,$ and fix a positive integer $k.$ Given a point $x \in M,$ 
define the function
$$\phi(R) = \frac{1}{R^{k} } \int_M \exp -\left(\dfrac{d \left( x, y \right)}{2R} \right)^2  \varphi^2_{x}(y) \, d \mu $$
and suppose that $\phi(R)$ is bounded above. Then
\begin{equation}\label{twolimits}
\lim_{R \searrow 0} \frac{\mu(B_R(x))}{ R^{k}} = \frac{\lim_{R \searrow 0} \phi(R)}{2^k \Gamma \left( \frac{k}{2} + 1 \right) }.
\end{equation}
Here $\Gamma$ is the Euler gamma function. 
\end{lemma}
\begin{proof} 
For simplicity, we will suppress the cutoff function $\varphi^2_{x}$ throughout the proof.

Define the increasing function
\begin{equation}
m(r) = \mu \left( B_r(x) \right).
\end{equation}
Since $\phi(r) \leq E_0 $ is bounded above, we clearly have
\begin{equation}\label{brbound}
0 \leq m(r) \leq e^{1/4} r^k \phi(r) \leq 2 E_0 r^k.
\end{equation}
For any $C^1$ radial function $f(r)$ with
\begin{equation}
|f(r)| + r |f'(r)| = O(r^{-k - 1}) \mbox{ as } r \to \infty
 \end{equation} 
we have
\begin{equation}\label{frformula}
\int \!\! f(r) \, d \mu = - \int_0^\infty \!\! f'(r) m(r) \, dr.
\end{equation}

We first assume that the limit on the LHS of (\ref{twolimits}) exists, so
\begin{equation}
m(r) \sim L r^k \mbox{ as } r \to 0
\end{equation}
for some $L > 0.$ Let $\psi(s) = \exp -s^2 / 4.$ Then, by (\ref{frformula}), we have
\begin{equation}
\begin{split}
\phi(R) & = \frac{1}{ R^{k} } \int_M \psi \left(\dfrac{r}{R} \right) \, d \mu  \\
& = - \int_0^\infty \psi'\left( \frac{r}{R} \right) \left( \frac{r}{R } \right)^k \frac{m(r)}{r^k}  \, \frac{dr}{R} \\
& = - \int_0^\infty \psi'\left( s \right) s^k \frac{m(Rs )}{(Rs)^k}  \, ds
\end{split}
\end{equation}
where $s = r / R.$ Hence
\begin{equation}\label{phipietc}
\begin{split}
\lim_{R \searrow 0} \phi(R) & = - L \int_0^\infty \psi'\left( s \right) s^k \, ds \\
& = k L \int_0^\infty \psi\left( s \right) s^{k - 1} \, ds \\
& = \frac{k L \left( 4 \pi \right)^{k/2}}{\omega_{k-1}}
\end{split}
\end{equation}
where $\omega_{k-1} = \mathrm{Vol}(S^{k-1}).$ 
Applying the formula for $\omega_{k-1}$ in 
(\ref{phipietc}), and rearranging, yields (\ref{twolimits}).

Next, we assume that the limit on the RHS of (\ref{twolimits}) exists, and show that the limit on the LHS exists using a Laplace-transform trick.

\vspace{2mm}

\noindent {\bf Claim 1.} Let $\chi(x)$ be the characteristic function of the unit interval $\LB 0, 1 \RB \subset \R_{\geq 0},$ and fix an integer $N_0 > \frac{k + 1}{2}.$ Given $\epsilon > 0,$ there exists an integer $N = N(\epsilon) > N_0$ and an approximating function $\tilde{\chi}_{\epsilon}(x),$ 
of the form
\begin{equation}\label{formofchitilde}
\tilde{\chi}_{\epsilon}(x) = \sum_{i = N_0}^N \frac{a_i }{(x + 1)^i}
\end{equation}
which 
satisfies
\begin{equation}\label{chiepsapprox}
\begin{split}
- \epsilon < \tilde{\chi}_{\epsilon}(x) & < 1 + \epsilon \quad  \left( 0 \leq x < \infty \right) \\
\left| \tilde{\chi}_{\epsilon}(x) - \chi(x) \right| & \leq \begin{cases} \epsilon & \left( 0 \leq x \leq 1 - \epsilon \right) \\
\dfrac{\epsilon}{x^{N_0}} & \left( 1 + \epsilon \leq x < \infty \right).
\end{cases}
\end{split}
\end{equation}

\begin{proof}[Proof of Claim 1] Let $u = \frac{1}{x + 1}.$ We use Weierstrass approximation on the unit interval in the $u$ variable, as follows. 
Let $0 \leq \sigma_\epsilon \leq 1$ be a continuous function satisfying
\begin{equation}
\sigma_\epsilon (u) = \begin{cases} 0 & \left(0 \leq u \leq \frac{1 - \epsilon}{2} \right) \\
1 & \left( \frac{1 + \epsilon}{2} \leq u \leq 1 \right).
\end{cases}
\end{equation}
We may let $p(u)$ be a polynomial satisfying
\begin{equation}
\left| p(u) - u^{-N_0} \sigma_\epsilon (u) \right| < \epsilon/2 \quad (0 \leq u \leq 1).
\end{equation}
Letting
$$\tilde{\chi}_{\epsilon}(u) = u^{N_0} p(u)$$
and substituting $u = \frac{1}{x + 1},$ we obtain a function $\tilde{\chi}_{\epsilon}(x)$ of the form (\ref{formofchitilde}) which satisfies (\ref{chiepsapprox}), as claimed.
\end{proof}

\vspace{2mm}

\noindent {\bf Claim 2.} Let
\begin{equation}\label{gepsdef}
g_\epsilon (\lambda) = - e^{-1/4\lambda^2} \lambda^{k} \sum_{i = N_0}^N \frac{a_i }{ \lambda^{2i + 1} 2^{2i-1} (i-1)! }.
\end{equation}
Then
\begin{equation*}
\tilde{\chi}_{\epsilon}(x) = \int_0^\infty g_\epsilon (\lambda) \frac{e^{-x / 4 \lambda^2}}{\lambda^{k} } \, d\lambda .
\end{equation*}
\begin{proof}[Proof of Claim 2]
We have the Laplace-transform identity
\begin{equation}
\frac{1}{\left( x + 1 \right)^i} = \int_{0}^\infty \frac{s^{i-1} e^{-s}}{(i-1)!}  e^{-xs} \, ds.
\end{equation}
Changing variables $s = \frac{1}{4 \lambda^2}$ yields the claim.
\end{proof}

\noindent {\bf Claim 3.} Let $x = r^2,$ and define
$$\chi_\epsilon(r) = \tilde{\chi_\epsilon}(r^2).$$
The limit
\begin{equation}
L_\epsilon = \lim_{R \searrow 0} R^{-k} \int \!\! \chi_\epsilon \left( \frac{r}{R} \right) \, d \mu
\end{equation}
exists, and is bounded independently of $\epsilon.$
\begin{proof}[Proof of Claim 3.] In the $r$ variable, the bounds (\ref{chiepsapprox}) become
\begin{equation}\label{chiepsrapprox}
\begin{split}
- \epsilon < \chi_{\epsilon}(r) & < 1 + \epsilon \quad  \left( 0 \leq r < \infty \right) \\
\left| \chi_{\epsilon}(r) - \chi(r) \right| & \leq \begin{cases} \epsilon & \left( 0 \leq r \leq 1 - \epsilon \right) \\
\dfrac{\epsilon}{r^{2 N_0}} & \left( 1 + \epsilon \leq r < \infty \right).
\end{cases}
\end{split}
\end{equation}
Since $2 N_0 > k + 1,$ the boundedness follows from 
(\ref{brbound}). To show that the limit exists for a given $\epsilon > 0,$ from Claim 2, we have 
\begin{equation}\label{claim3limit}
\begin{split}
R^{-k} \int \!\!\! \chi_\epsilon \left( \frac{r}{ R } \right) \, d \mu & = R^{-k} \int \!\!\!\!\! \int_0^\infty g_\epsilon (\lambda) \frac{e^{-r^2 / 4 R^2 \lambda^2}}{\lambda^{k} } \, d\lambda \, d \mu \\
& = \int_0^\infty \!\!\! g_\epsilon (\lambda) \int \frac{e^{-r^2 / 4 R^2 \lambda^2}}{(R\lambda)^{k} } \, d \mu \, d\lambda \\
& = \int \!\!\! g_\epsilon (\lambda) \phi(\lambda R) \, d\lambda.
\end{split}
\end{equation}
By assumption, $\phi(\cdot)$ is continuous at zero and bounded, while $g_\epsilon(\lambda)$ is absolutely integrable. Hence (\ref{claim3limit}) yields 
$$\lim_{R \to 0} R^{-k} \int \!\!\! \chi_\epsilon \left( \frac{r}{ R } \right) \, d \mu = \left( \lim_{s \searrow 0}\phi(s) \right) \int_0^\infty g_\epsilon ( \lambda)\, d \lambda.$$
This proves the claim.
\end{proof}

\noindent {\bf Claim 4.} We have
\begin{equation}\label{claim4}
\int \!\!\! \chi_\epsilon \left( \frac{r}{(1 - \epsilon) R} \right) \, d \mu - C \epsilon R^{k} \leq \mu(B_R) \leq \int \!\!\! \chi_\epsilon \left( \frac{r}{ (1 + \epsilon) R} \right) \, d \mu + C \epsilon R^{k}.
\end{equation}
\begin{proof}[Proof of Claim 4]
From (\ref{chiepsrapprox}), we have
\begin{equation}\label{chiepslower}
 \chi_\epsilon \left( \frac{r}{ (1 - \epsilon) R} \right) \leq \chi \left( \frac{r}{R} \right) +  \epsilon \left( \frac{ 2 R}{R + r } \right)^{2 N_0} 
\end{equation}
and
\begin{equation}\label{chiepsupper}
\chi \left( \frac{r}{R} \right) \leq \chi_\epsilon \left( \frac{r} { (1 + \epsilon) R} \right) + \epsilon \left( \frac{ 2 R}{R + r } \right)^{2 N_0}.
\end{equation}
From (\ref{brbound}) and (\ref{frformula}), we have
\begin{equation}\label{epsbound}
\begin{split}
\int \!\! \epsilon \left( \frac{ 2 R}{R + r } \right)^{2 N_0}  \, \!\! d \mu & \leq C \epsilon \int_0^{\infty} \frac{\left( 2 R\right)^{2 N_0 + k} }{(R + r)^{2 N_0 + 1} } \, dr \\
& \leq C \epsilon R^k.
\end{split}
\end{equation}
Integrating (\ref{chiepslower}) and (\ref{chiepsupper}) in $r$ and applying (\ref{epsbound}) yields the claim.
\end{proof}

To complete the proof of Lemma \ref{lemma:laplacelemma}, let $L_\epsilon$ be the limit in Claim 3. Dividing (\ref{claim4}) by $R^k$ and taking the limit as $R \searrow 0,$ we obtain
\begin{equation}\label{oneguy}
(1 - \epsilon)^k L_\epsilon - C \epsilon \leq \liminf_{R \searrow 0} \frac{\mu(B_R)}{R^k}
\end{equation}
and
\begin{equation}\label{twoguy}
\limsup_{R \searrow 0} \frac{\mu(B_R)}{R^k} \leq (1 + \epsilon)^k L_\epsilon + C \epsilon.
\end{equation}
Subtracting (\ref{oneguy}) from (\ref{twoguy}) yields
\begin{equation*}
\limsup_{R \searrow 0} \frac{\mu(B_R)}{R^k} - \liminf_{R \searrow 0} \frac{ \mu(B_R) }{R^k} \leq C \left( 1 + | L_\epsilon | \right) \epsilon.
\end{equation*}
Since $\epsilon > 0$ was arbitrary and $L_\epsilon$ is bounded, we conclude that the limit on the LHS of (\ref{twolimits}) exists.

By the first part of the proof, the two limits must again satisfy (\ref{twolimits}).
\end{proof}

\begin{lemma}\label{lemma:phisupbound} Let $\mu$ and $\phi(r)$ be as in Lemma \ref{lemma:laplacelemma}, and assume $\phi(r) \leq E_0$ for all $r > 0.$ If, for some $R > 0$ and $\epsilon \leq 1,$ we have
\begin{equation}\label{phisupbound:mRass}
m(R) = \mu(B_R) \leq \epsilon R^k
\end{equation}
then
\begin{equation}\label{phisupbound:est}
\phi \left( \epsilon^{1/2k} R \right) \leq C_{k,n} \left( \sqrt{\epsilon} + E_0 \exp \left( - \epsilon^{-1/{2k}} \right) \right).
\end{equation}
\end{lemma}
\begin{proof}  Let $0 \leq \alpha \leq 1,$ and put $R_1 = \alpha R.$ Notice from (\ref{phisupbound:mRass}) that for $R_1 \leq r \leq R,$ we have
\begin{equation}
\begin{split}
\frac{m(r)}{r^k} = \left(\frac{R}{r} \right)^k R^{-k} m(r) & \leq \left(\frac{R}{R_1} \right)^k R^{-k} m(R) \\
 & \leq \alpha^{-k} \epsilon.
\end{split}
\end{equation}
Let $\psi(x) = \exp -x^2 / 4$ as above. Then, by (\ref{frformula}), we have 
\begin{equation}
\begin{split}
\phi(R_1) & \leq \frac{1}{ R_1^{k} } \int_M \psi \left(\dfrac{r}{R_1} \right) \, d \mu  \\
& = \frac{1}{ R_1^{k} } \int_{B_{R_1}} \psi \left(\dfrac{r}{R_1} \right) \, d \mu - \frac{\psi(R_1) m(R_1) }{R_1^k} \\
& \qquad - \left( \int_{R_1}^R + \int_R^\infty \right) \psi'\left( \frac{r}{R_1} \right) \left( \frac{r}{R_1 } \right)^k \frac{m(r)}{r^k} \, \frac{dr}{R_1} \\
& \leq C \alpha^{-k} \epsilon \left( 1 + \int_{1}^{\alpha^{-1}} \left| \psi'\left( s \right) \right| s^k ds \right) + C E_0 \int_{\alpha^{-1}}^\infty \left| \psi'\left( s \right) \right| s^k ds
\end{split}
\end{equation}
where we have let $s = r / R_1.$ This gives
\begin{equation}
\phi(R_1) \leq C \left( \alpha^{-k} \epsilon + E_0 \exp - \left( \alpha^{-2} / 5 \right) \right).
\end{equation}
Letting $\alpha = \epsilon^{1/2k}$ yields the claim.
\end{proof}

\subsection{Hausdorff-measure estimates} We collect here the Hausdorff estimates which will be used in the proof of Theorems \ref{thm:sigma}-\ref{thm:uhlenbeck}.

For a sequence of solutions $\{A_i\}$ as in Theorem \ref{thm:sigma}, let $\Phi_i$ and $\Xi_i$ denote the quantities (\ref{phidef}-\ref{xidef}) corresponding to $A = A_i.$ Write
\begin{equation}\label{liminfdef}
\underline{\Phi} = \liminf_{i \to \infty} \Phi_i, \qquad \underline{\Xi} = \liminf_{i \to \infty} \Xi_i.
\end{equation}

\begin{prop}\label{prop:hn-4} For $\Sigma$ as in Theorem \ref{thm:sigma}, we have $\mathcal{H}^{n-4}(\Sigma \cap U) < \infty$ for any $U \Subset M.$
\end{prop}
\begin{proof} 
Without loss of generality, we replace $\Sigma$ by $\Sigma \cap U$ in the proof. By (\ref{sigma:finitenergy}), we may assume (\ref{efiniteenergy}).

Notice, from Hamilton's monotonicity formula (\ref{hamilton:epsilonestimate}), that there exists $R_0 > 0$ such that for any $ 0 < R_1 < R_0,$ and for every $x \in \Sigma,$ we have
\begin{equation}\label{hn-4:eps/2}
\underline{\Phi} \left(R_1, x, \tau - R_1^2 \right) \geq \frac{\epsilon_0}{2}.
\end{equation}

Let $\epsilon > 0$ be such that the RHS of (\ref{phisupbound:est}) is less than $\epsilon_0/2,$ and let
\begin{equation}\label{hn-4:r'def}
R = \frac{R_1}{ \epsilon^{1/2 (n-4)} }.
\end{equation}
The contrapositive of Lemma \ref{lemma:phisupbound} and (\ref{hn-4:eps/2}) imply 
\begin{equation}\label{hn-4:r'est}
\liminf_{i \to \infty} \int_{B_{R}(x)} \left| F_i \left(\tau - R_1^2 \right) \right|^2 \, dV \geq \epsilon {R}^{n-4}
\end{equation}
for any $x \in \Sigma.$

We now estimate the Hausdorff measure by the argument of Nakajima. 
By the Vitali covering lemma, 
we may let $\{x_k \} \subset \Sigma $ be such that $ \Sigma \subset \cup B_{5R}(x_k)$ and $B_{R}(x_k) \cap B_{R}(x_\ell) = \emptyset$ for $k \neq \ell.$ 
We then have
\begin{equation*}
\begin{split}
\mathcal{H}^{n-4}_{5R}(\Sigma) \leq \sum_k (5R)^{n-4} & \leq \frac{5^{n-4}}{\epsilon} \sum_k \liminf_{i \to \infty} \int_{B_{R}(x_k)} \left| F_i \left( \tau - R_1^2 \right) \right|^2 \, dV \\
& \leq \frac{5^{n-4}}{\epsilon} \liminf_{i \to \infty} \sum_k  \int_{B_{R}(x_k)} \left| F_i \left( \tau - R_1^2 \right) \right|^2 \, dV \\
& \leq \frac{5^{n-4}}{\epsilon} \liminf_{i \to \infty} \int_M \left| F_i \left( \tau - R_1^2 \right) \right|^2 \, dV \\
& \leq \frac{5^{n-4} E}{\epsilon}.
\end{split}
\end{equation*}
Since $R$ tends to zero with $R_1$ (by (\ref{hn-4:r'def})), we are done.
\end{proof}

\begin{prop}\label{prop:dumbliminf} For $\tau_i \nearrow \tau,$ let $f(x) = \liminf_{i \to \infty} \left| F_i \left(x, \tau_i \right) \right|^2.$ Then
\begin{equation}\label{dumbliminf:est}
\lim_{R \searrow 0}  \int_M f(y) u_{R,x}(y) \varphi^2_{x}(y) \, dV_y = 0
\end{equation}
for all $x \in M \setminus \Sigma$ and for $\mathcal{H}^{n-4}$-a.e. $x \in \Sigma.$
\end{prop}
\begin{proof} First, note from the $\epsilon$-regularity theorem that $f(x)$ is locally bounded on $M \setminus \Sigma.$ Hence the limit (\ref{dumbliminf:est}) is zero if $x \in M \setminus \Sigma.$

We may again replace $\Sigma$ by $\Sigma \cap U$ for an open subset $U \Subset M.$ Let $S_j\subset \Sigma$ be the set of $x$ such that
\begin{equation}\label{dumbliminf:rest}
\limsup_{R_1 \searrow 0} \int_M f(y) u_{R_1, x}(y) \varphi^2_{x, \rho_1}(y) \, dV_y \geq \frac{1}{j}.
\end{equation}
We will show that $\mathcal{H}^{n-4} (S_j) = 0.$

Let $\delta > 0.$ Define $\epsilon_j > 0$ such that the RHS of (\ref{phisupbound:est}), with $\epsilon = \epsilon_j,$ is equal to $1/2j.$ By the contrapositive of Lemma \ref{lemma:phisupbound}, (\ref{dumbliminf:est}) implies that for each $x \in S_j,$ there exists $0 < R_x < \delta$ such that 
\begin{equation}\label{dumbliminf:rxest}
R_x^{4-n} \int_{B_{R_x}(x)} f(y) \, dV_y \geq \epsilon_j.
\end{equation}
Let $\{x_k \} \subset \Sigma$ be such that $ \Sigma \subset \cup B_{5R_{x_k}}(x_k)$ and $B_{R_{x_k}}(x_k) \cap B_{R_{x_\ell}}(x_\ell) = \emptyset$ for $k \neq \ell.$ By (\ref{dumbliminf:rxest}), we have
\begin{equation}
\begin{split}
\mathcal{H}^{n-4}_{5\delta}(S_j) \leq 5^{n-4}\sum_k R_{x_k}^{n-4} & \leq \frac{5^{n-4}}{\epsilon_j} \sum_k \int_{B_{R_{x_k}}(x_k)} f(y) \, dV_y \\
& \leq \frac{5^{n-4}}{\epsilon_j} \int_{\Sigma_\delta} f(y) \, dV_y = \frac{5^{n-4}}{\epsilon_j} \int_M \chi_{_{\Sigma_\delta}} f(y) \, dV_y.
\end{split}
\end{equation}
Since $\Sigma$ is closed and of Lebesgue measure zero, $\chi_{_{\Sigma_\delta}} f(y) \to 0$ pointwise almost everywhere as $\delta \to 0.$ Hence, 
the last integral tends to zero by the dominated convergence theorem. This completes the proof that $\mathcal{H}^{n-4}(S_j) = 0.$

The set of $x$ satisfying (\ref{dumbliminf:est}) is the complement of the union of $S_j,$ for $j = 1, \ldots , \infty,$ and therefore has full $\mathcal{H}^{n-4}$-measure in $\Sigma.$
\end{proof}

\begin{prop}\label{prop:xiprop} Assuming (\ref{sigma:dstarfass}), we have
\begin{equation}\label{xiprop:est}
\lim_{\sigma \nearrow \tau} \limsup_{R \searrow 0} \underline{\Xi} \left(R, x, \sigma, \tau \right) = 0
\end{equation}
for all $x \in M \setminus \Sigma$ and for $\mathcal{H}^{n-4}$-a.e. $x \in \Sigma.$ Here $\underline{\Xi}$ is defined by (\ref{liminfdef}).
\end{prop}
\begin{proof} The proof is similar to that of the previous Proposition, with an extra step. First, note from Lemma \ref{lemma:basicepsreg} that the limit (\ref{xiprop:est}) is zero if $x \in M \setminus \Sigma.$ 

Let $S_j$ be the set of points $x \in \Sigma$ such that
\begin{equation*}
\limsup_{\sigma \nearrow \tau} \limsup_{R \searrow 0} \underline{\Xi} \left( R, x, \sigma, \tau \right) \geq \frac{1}{j}.
\end{equation*}
For each $k > 0,$ by (\ref{sigma:dstarfass}), we may choose $\sigma_k < \tau$ such that
\begin{equation*}
\liminf_{i \to \infty} \int_{\sigma_k}^\tau \!\!\! \int_M |D^*F_i|^2 \, dV dt \leq \frac{1}{k}.
\end{equation*}
Further let $S_{j,k} \supset S_j$ be the set of $x \in \Sigma$ such that
\begin{equation}\label{xiprop:rest}
\limsup_{R \searrow 0} \underline{\Xi} (R, x, \sigma_k, \tau) \geq \frac{1}{j}.
\end{equation}
We will show that $\mathcal{H}^{n-4} (S_{j,k}) \to 0$ as $k \to \infty,$ for each fixed $j.$ Since $S_j \subset S_{j,k},$ this will imply that $\mathcal{H}^{n-4} (S_j) = 0.$

Let $\delta > 0.$ As above, define $\epsilon_j > 0$ such that the RHS of (\ref{phisupbound:est}), with $\epsilon = \epsilon_j,$ is equal to $1/2j.$ By the contrapositive of Lemma \ref{lemma:phisupbound}, (\ref{xiprop:rest}) implies that for each $x \in S_{j,k},$ there exists $0 < R_x < \delta$ such that 
\begin{equation}\label{xiprop:rxest}
R_x^{4-n} \liminf_{i \to \infty} \int_{\sigma_k}^\tau \!\!\! \int_{B_{R_x}(x)} |D^*F_i|^2 \, dV dt \geq \epsilon_j.
\end{equation}
Let $\{x_\ell \} \subset S_{j,k}$ be such that $ S_{j,k}\subset \cup B_{5R_{x_\ell}}(x_\ell)$ and $B_{R_{x_\ell}}(x_\ell) \cap B_{R_{x_m}}(x_m) = \emptyset$ for $\ell \neq m.$ By (\ref{xiprop:rxest}), we have
\begin{equation*}
\begin{split}
\mathcal{H}^{n-4}_{5\delta}(S_{j,k}) \leq \sum_\ell (5 R_{x_\ell})^{n-4} & \leq \frac{5^{n-4}}{\epsilon_j} \sum_\ell \liminf_{i \to \infty} \int_{\sigma_k}^\tau \!\!\! \int_{B_{R_{x_\ell}}(x_\ell)} |D^*F_i|^2 \, dV dt \\
& \leq \frac{5^{n-4}}{\epsilon_j}  \liminf_{i \to \infty} \sum_\ell \int_{\sigma_k}^\tau \!\!\! \int_{B_{R_{x_\ell}}(x_\ell)} |D^*F_i|^2 \, dV dt \\
& \leq \frac{5^{n-4}}{\epsilon_j} \liminf_{i \to \infty} \int_{\sigma_k}^\tau \!\!\! \int_{M} |D^*F_i|^2 \, dV dt \\
& \leq \frac{5^{n-4}}{k \epsilon_j}.
\end{split}
\end{equation*}
We may let $\delta \to 0,$ then let $k \to \infty,$ to conclude that
$$\lim_{k \to \infty} \mathcal{H}^{n-4}(S_{j,k}) = 0$$
and therefore $\mathcal{H}^{n-4}(S_j) = 0.$

As before, the set of $x$ satisfying (\ref{dumbliminf:est}) is the complement of the union of $S_j,$ for $j = 1,2, \ldots ,$ and therefore has full $\mathcal{H}^{n-4}$-measure in $\Sigma.$
\end{proof}

\begin{cor}\label{cor:xicor} Assuming (\ref{sigma:dstarfass}), for any $L > 0,$ we have
\begin{equation}\label{xicor:est}
\lim_{R \searrow 0} \underline{\Xi} \left(R, x, \tau - L R^2, \tau \right) = 0
\end{equation}
for all $x \in M \setminus \Sigma$ and for $\mathcal{H}^{n-4}$-a.e. $x \in \Sigma.$
\end{cor}
\begin{proof} For a given $x,$ the condition (\ref{xicor:est}) is implied by (\ref{xiprop:est}).
\end{proof}

\subsection{Gauge-patching lemmas}

This section carries out a minor correction to Lemmas 4.4.5-4.4.7 and Corollary 4.4.8 of Donaldson and Kronheimer \cite{donkron}, originally due to Uhlenbeck \cite{uhlenbecklp} in a different form.\footnote{Note the phrase ``we extend $\left. \tilde{u} \right|_N$ arbitrarily over $\Omega$...'' in the middle of p. 159 of \cite{donkron}. To see that it is not always possible to extend gauge transformations, consider the trivial $SU(2)$ bundle on $B_1 \subset \R^4$ and the gauge transformation defined over $B_1 \setminus B_{1/2}$ corresponding to the identity map of $SU(2) \cong S^3.$} 

\begin{rmk}\label{rmk:referenceconn} As in Lemma \ref{lemma:distancecomp}, we fix a reference connection $\nabla_{ref}$ on the bundle $E,$ which we use to define the $C^k$ norms on bundle-valued differential forms. By definition, the $C^k$ norm of a connection is equal to $\| A\|_{C^k},$ where $A$ is the unique (global) 1-form such that $\nabla_A = \nabla_{ref} + A.$
\end{rmk}

\begin{lemma}[\textit{Cf.} \cite{donkron}, Lemma 4.4.5]\label{lemma:dkunfuck} Suppose that $A_i$ is a sequence of connections on a bundle $E$ over a base manifold $\Omega,$ and $\tilde{\Omega} \Subset \Omega$ is an interior domain. Suppose that there are gauge transformations $u_i \in \mbox{Aut} \, E$ and $\tilde{u}_i \in \left. \mbox{Aut} \, E \right|_{\tilde{\Omega}}$ such that $u_i(A_i)$ converges in $C^\infty_{loc}\left(\Omega \right)$ and $\tilde{u}_i (A_i)$ converges over $C^\infty_{loc}\left( \tilde{\Omega} \right).$ Then for any compact set $K \subset \tilde{\Omega},$ there exists a subsequence $\{j \} \subset \{i\},$ with $j \geq j_0,$ and gauge transformations $w_j$ such that
\begin{equation}
w_j = \begin{cases}
\tilde{u}_{j_0}^{-1} \tilde{u}_j & \mbox{ on } K \\
u_{j_0}^{-1} u_j &  \mbox{ on } \Omega \setminus \tilde{\Omega}
\end{cases}
\end{equation}
and the connections $w_j(A_j)$ converge in $C^\infty_{loc}\left(\Omega \right).$
\end{lemma}
\begin{proof} Define the gauge transformations over $\tilde{\Omega}$
\begin{equation*}
v_i = \tilde{u}_i u_i^{-1}.
\end{equation*}
These satisfy
\begin{equation}\label{dkunfuck:relativegauge}
v_i(u_i(A_i)) = \tilde{u}_i(A_i).
\end{equation}
Choose an open set $N$ with
$$K \subset N \Subset \tilde{\Omega}.$$
Since both $u_i(A_i)$ and $\tilde{u}_i(A_i)$ are smoothly convergent, we conclude from the usual bootstrapping argument (e.g. \cite{donkron}, p. 64) applied to (\ref{dkunfuck:relativegauge}), that $v_i$ are bounded in $C^k(N)$ for all $k.$ By the Arzela-Ascoli Theorem, we may extract a convergent subsequence $v_j.$ Choosing $j_0$ sufficiently large, we may assume
\begin{equation*}
z_j = v_{j_0}^{-1} v_j
\end{equation*} 
is arbitrarily close to the identity on $N$ for all $j \geq j_0.$ Letting $\xi_j = \log z_j,$ and choosing a cutoff $\psi$ for $K \subset N,$ we may extend $z_j$ over $\Omega$ by the formula
$$z_j = \exp \left(\psi \xi_j \right).$$
Defining
$$w_j = u_{j_0}^{-1} z_j u_j$$
yields the desired sequence of gauge transformations over $\Omega.$
\end{proof}

\begin{lemma}[\textit{Cf.} \cite{donkron}, Lemma 4.4.7]\label{lemma:dktwopatch} Suppose $\Omega$ is a union of domains $\Omega = \Omega_1 \cup \Omega_2,$ and $A_i$ is a sequence of connections on a bundle $E$ over $\Omega.$ Choose a compactly contained subdomain $\Omega' \Subset \Omega.$ If there are sequences of gauge transformations $v_i \in \left. \mbox{Aut} \, E \right|_{\Omega_1}$ and $w_i \in \left. \mbox{Aut} \, E \right|_{\Omega_2}$ such that $v_i(A_i)$ and $w_i(A_i)$ converge over $\Omega_1$ and $\Omega_2,$ respectively, then there is a subsequence $\{j \}$ and gauge transformations $u_j$ over $\Omega'$ such that $u_j(A_j)$ converges over $\Omega'.$
\end{lemma}
\begin{proof} We may assume without loss that $\Omega' \subset \Omega_1' \cup \Omega_2',$ for subdomains $\Omega_1' \Subset \Omega_1$ and $\Omega_2' \Subset \Omega_2.$ Applying Lemma \ref{lemma:dkunfuck} with $\Omega = \Omega_1$ and $K = \Omega_1' \cap \Omega_2',$ we obtain gauge transformations $v'_j$ over $\Omega_1$ such that $v'_j = w_{j_0}^{-1} w_j$ over $\Omega_1'.$ Then each $v'_j$ glues together with $w_{j_0}^{-1} w_j$ to define gauge transformations $u_j$ over $\Omega'$ such that $u_j(A_j)$ converges, as desired.
\end{proof}

\begin{cor}[\textit{Cf.} \cite{donkron}, Corollary 4.4.8]\label{cor:dkballs} Suppose $A_i$ is a sequence of connections on $E \to \Omega$ with the following property. For each point $x \in \Omega,$ and each subsequence $\{j \} \subset \{i\},$ there is a neighborhood $D$ of $x,$ a further subsequence $\{k \} \subset \{j \},$ and gauge transformations $v_k$ defined over $D$ such that $v_k(A_k)$ converges over $D.$ Then for any subdomain $\Omega' \Subset \Omega,$ there is a single subsequence $\{j\}$ and gauge transformations $u_j$ defined over $\Omega',$ such that $u_j(A_j)$ converges over all of $\Omega'.$
\end{cor}
\begin{proof} Choose a finite cover of $\Omega'$ by balls $D_\ell \Subset \Omega$ with the stated property. Applying Lemma \ref{lemma:dktwopatch} with $\Omega_1 = D_1 \cup \cdots \cup D_{m - 1}$ and $\Omega_2 = D_m,$ we may cobble together the gauges on $D_\ell,$ in finitely many steps, to obtain a sequence of global gauges on $\Omega'$ with the desired property.
\end{proof}

\begin{cor}[\textit{Cf.} \cite{donkron}, Lemma 4.4.6]\label{cor:exhaustion} Suppose that $\Omega$ is exhausted by an increasing sequence of precompact open sets
\begin{equation}
U_1 \Subset U_2 \Subset \cdots \subset \Omega, \quad \bigcup_{m = 1}^{\infty} U_m = \Omega.
\end{equation}
Let $A_i$ be a sequence of connections on $E \to \Omega$ with the property stated in Corollary \ref{cor:dkballs}. 
Then there is a subsequence $\{j\},$ a bundle $E_\infty \to \Omega,$ and bundle maps
\begin{equation}\label{exhaustion:ujmap}
u_j : \left. E \right|_{U_j} \to \left. E_\infty \right|_{U_j}
\end{equation}
such that $u_j(A_j)$
converges in $C^\infty_{loc}(\Omega)$ to a connection $A_\infty$ on $E_\infty \to \Omega.$
\end{cor}
\begin{proof} By Corollary \ref{cor:dkballs} and a diagonal argument, we can pass to a subsequence of $A_i$ such that for each $m,$ there exist gauge transformations $v_i^m$ defined over $U_m,$ for $i \geq m,$ such that $v^m_i(A_i)$ converges in $C^\infty(U_m)$ as $i \to \infty.$

From the $v^m_i,$ we construct new sequences $w^m_j,$ for $j \geq j_m,$ as follows. Let $w^1_i = v^1_i,$ for $i \geq 1.$ Assuming that $w^m_j$ has been constructed, by applying Lemma \ref{lemma:dkunfuck} to the sequence $w^m_j$ and $v^{m+1}_j,$ we may obtain a sequence $w^{m + 1}_j$ on $U_{m + 1},$ for $j \geq j_{m + 1},$ such that
\begin{equation}\label{exhaustion:wu}
\left. w^{m + 1}_j \right|_{U_m} = \left( w^m_{j_{m+1}} \right)^{-1} w^m_j
\end{equation}
and $w^{m +1}_j(A_j)$ converges on $U_{m + 1}.$ 

We now define the bundle $E_\infty$ as having transition functions $g_{(m + 1)m} = w^m_{j_{m + 1}} $ from $U_{m + 1}$ to $U_{m}.$ We may then define $u_m$ to equal $w^m_{j_m}$ on $U_m,$ which gives a well-defined bundle map of the form (\ref{exhaustion:ujmap}). By (\ref{exhaustion:wu}), the images $u_m(A_m)$ converge on $\left. E_\infty \right|_{U_n}$ for each $n.$
\end{proof}

\begin{rmk} In the case that $U_i$ is a deformation retract of $\Omega$ for sufficiently large $i,$ we may take $E_\infty = \left. E \right|_{\Omega},$ and the $u_i$ may be assumed to be defined over all of $\Omega$ (see Wehrheim \cite{wehrheimuhlenbeck}). Such is the case in dimension four, where $\Sigma$ is a finite set of points, and in the K\"ahler situation in higher dimensions, where $\Sigma$ is a holomorphic subvariety \cite{sibleyasymptotics}.
\end{rmk}

\section{Proof of Theorem \ref{thm:sigma}}

\begin{proof}[Proof of Theorem \ref{thm:sigma}] As above, we let $\Phi_i$ and $\Xi_i$ denote the quantities (\ref{phidef}-\ref{xidef}) corresponding to $A = A_i,$  and define $\underline{\Phi}$ and $\underline{\Xi}$ by (\ref{liminfdef}). Fix $U \Subset U_1 \Subset M$ satisfying (\ref{brho1}), and let
$$\rho_0 = \inf_{x \in U} \rho_1(x).$$
Then, it suffices to prove closedness and rectifiability of $\Sigma \cap U \subset U.$ We replace $\Sigma$ by $\Sigma \cap U$ for the remainder of the proof.

Note that (\ref{sigma:finitenergy}) implies a bound of the form (\ref{efiniteenergy}) for all $A_i,$ with a uniform $E > 0.$ Then, by Hamilton's monotonicity formula (\ref{hamilton:epsilonestimate}), we have a uniform bound
\begin{equation}\label{sigma:e0bound}
\Phi_i(R, x, t) \leq C E = : E_0
\end{equation}
for all $0 < R < \rho_0,$  $x \in U,$  $\rho_0^2 \leq t < \tau,$ and all $i.$ 

Closedness of $\Sigma $ follows by adapting the argument of Nakajima \cite{nakajimacompactness}, as follows. Suppose $\{x_j\} \subset \Sigma$ is a sequence converging to $x \in U.$ 
Let $\epsilon > 0,$ and choose $R > 0$ sufficiently small that
\begin{equation}\label{sigma:2}
R^{4 - n} \exp -\left(\frac{\rho_0}{4R} \right)^2 < \epsilon.
\end{equation}
We may fix $j$ sufficiently large that 
\begin{equation}
 \exp -\left(\dfrac{d \left( x, y \right)}{2R} \right)^2  \varphi_{x}(y) \geq \left( 1 - \epsilon \right) \exp -\left(\dfrac{d \left( x_j, y \right)}{2R} \right)^2  \varphi_{x_j}(y) - C \exp -\left(\dfrac{\rho_0}{4R} \right)^2.
\end{equation}
Integrating against $R^{4-n} |F_i(\tau - R^2)|^2$ yields
\begin{equation}\label{sigma:1}
\begin{split}
\Phi_i(R, x, \tau - R^2) & \geq \left( 1 - \epsilon \right) \Phi_i(R, x_j, \tau - R^2) - C \left( R^{4 - n} \exp -\left(\dfrac{\rho_0}{4R} \right)^2 \right) E.
\end{split}
\end{equation}

Now, because $x_j \in \Sigma,$ there exists $0 < R' \leq R$ such that for all sufficiently large $i,$ we have
\begin{equation*}
\Phi_i(R', x_j, \tau - R'^2) \geq \epsilon_0 - \epsilon.
\end{equation*}
Applying the monotonicity formula (\ref{hamilton:epsilonestimate}) yields
\begin{equation}\label{sigma:3}
\Phi_i(R, x_j, \tau - R^2) \geq \epsilon_0 - 2\epsilon
\end{equation}
provided that $R < R_0.$
Applying $\liminf$ to both sides of (\ref{sigma:1}), and inserting (\ref{sigma:2}) and (\ref{sigma:3}) yields
\begin{equation}\label{sigma:4}
\begin{split}
\underline{ \Phi} (R, x, \tau - R^2) & \geq \left( 1 - \epsilon \right) \left(\epsilon_0 - 2 \epsilon \right) - C \epsilon E \\
& \geq \epsilon_0 - C \epsilon (\epsilon_0 + E).
\end{split}
\end{equation}
Since $\epsilon$ was arbitrary, we conclude
\begin{equation}
\liminf_{R \searrow 0} \underline{ \Phi} (R, x, \tau - R^2)  \geq \epsilon_0
\end{equation}
as desired.

This completes the proof that $\Sigma$ is closed. Local finiteness of the $\mathcal{H}^{n-4}$-measure is shown in Proposition \ref{prop:hn-4}.

Next, by weak compactness of locally uniformly bounded measures (\ref{sigma:finitenergy}), we may pass to a subsequence such that the limit of measures
\begin{equation}\label{sigma:muexists}
\mu = \lim_{i \to \infty} \left( \left| F_i \left(\tau_i\right) \right|^2 \, dV \right)
\end{equation}
exists. By Fatou's Lemma, 
we may then write
\begin{equation}\label{sigma:nudef}
\mu = \left( \liminf_{i \to \infty}  \left| F_i \left( \tau_i \right) \right |^2 \right) \, dV + \nu
\end{equation}
where $\nu$ is a nonnegative measure supported on $\Sigma.$

To show rectifiability assuming (\ref{sigma:dstarfass}), 
we claim that
\begin{equation}\label{sigma:phiexists}
\begin{split}
\phi(x) & = \lim_{R \searrow 0} R^{4 - n} \int_M \exp -\left(\dfrac{d \left( x, y \right)}{2R} \right)^2  \varphi^2_{x}(y) \, d \nu
\end{split}
\end{equation}
exists and is nonzero for $\mathcal{H}^{n-4}$-a.e. $x \in \Sigma.$ 

Let $\epsilon > 0.$ 
By (\ref{sigma:nudef}) and Proposition \ref{prop:dumbliminf}, we may replace $d\nu$ with $d \mu$ in the limit (\ref{sigma:phiexists}) for $\mathcal{H}^{n-4}$-a.e. $x \in \Sigma.$ Then
 (\ref{sigma:phiexists}) becomes
\begin{equation}\label{sigma:phireallyexists}
\phi(x) = \lim_{R \searrow 0} \lim_{i \to \infty} \Phi_i \left(R, x, \tau_i \right)
\end{equation}
where the inner limit exists by (\ref{sigma:muexists}). Given (\ref{sigma:dstarfass}), by Corollary \ref{cor:xicor}, we may assume that $x$ is a point such that
\begin{equation}\label{sigma:11}
\lim_{R \searrow 0} \underline{\Xi} \left(R, x, \tau - R^2, \tau \right) = 0
\end{equation}
where $\underline{\Xi}$ is defined by (\ref{liminfdef}).

Then, for $R > 0$ be sufficiently small, 
there exists an infinite subsequence of integers $\{j\}$ such that 
\begin{equation*}
\Xi_j(R, x, \tau - R^2, \tau) < \epsilon^2.
\end{equation*}
By Lemma \ref{lemma:antibubble}, this implies that 
\begin{equation}\label{sigma:phicompar}
\left| \Phi_j \left(R, x, \tau_j \right) - \Phi_j \left(R, x, t \right)\right| \leq C \epsilon
\end{equation}
for any $\tau_j - R^2 \leq t \leq \tau_j.$

For any $0 < R' < R,$ we may apply the monotonicity formula (\ref{hamilton:epsilonestimate}) with $R_1 = R$ and $R_2 = R',$ to obtain
\begin{equation*}
\begin{split}
\Phi_j \left(R', x, \tau_j \right) & \leq \left( 1 + \epsilon \right) \Phi_j \left(R, x, \tau_j - R^2 + R'^2 \right) + C \epsilon.
\end{split}
\end{equation*}
Inserting (\ref{sigma:e0bound}) and (\ref{sigma:phicompar}), we have
\begin{equation*}
\Phi_j \left(R', x, \tau_j \right) \leq \Phi_j \left(R, x, \tau_j \right) + \left( C + E_0 \right) \epsilon.
\end{equation*}
This demonstrates that for $R$ sufficiently small, and any $0 < R' < R,$ we in fact have
\begin{equation}\label{sigma:almostdecreasing}
\lim_{i \to \infty} \Phi_i \left(R', x, \tau_i \right) \leq \lim_{i \to \infty} \Phi_i \left(R, x, \tau_i \right) + C\epsilon.
\end{equation}
Therefore
\begin{equation*}
\limsup_{R \searrow 0}\lim_{i \to \infty} \Phi_i \left(R, x, \tau_i \right) \leq \liminf_{R \searrow 0} \lim_{i \to \infty} \Phi_i \left(R, x, \tau_i \right) + C\epsilon.
\end{equation*}
Since $\epsilon > 0$ was arbitrary, 
this implies that the limit $R \searrow 0$ in (\ref{sigma:phireallyexists}) exists, as claimed.

We may conclude from (\ref{sigma:phiexists}), Lemma \ref{lemma:laplacelemma}, and Preiss's Theorem \cite{preiss} (stated as Theorem 1.1 of \cite{delellisrectifiable}) that the measure $\nu$ is $(n-4)$-rectifiable. By (\ref{sigma:nudef}) and Proposition \ref{prop:dumbliminf}, $\Sigma = \mbox{supp } \nu$ up to measure zero, hence the same is true of $\Sigma.$
\end{proof}

\section{Proof of Theorem \ref{thm:uhlenbeck} and Corollary \ref{cor:uhlenbeck}}

\begin{proof}[Proof of Theorem \ref{thm:uhlenbeck}] 
To construct the required subsequence and exhaustion, we argue as follows.

Given an open subset $U _* \subset \left( M \setminus \Sigma \right) \times \N$ (with the box topology), write
$$\mathcal{I}(U_*) = \pi_2(U_*).$$ 
Consider the collection of open subsets
\begin{equation}\label{uhlenbeck:zornset}
\mathcal{S} \subset \{ U_* \subset \left( M \setminus \Sigma \right) \times \N \mbox{ open} \}
\end{equation} 
which satisfy the following conditions: for all $i,j \in \mathcal{I}(U_*),$ there hold
\begin{equation}\label{uhlenbeck:zorncondition}
\sup_{\stackrel{x \in U_i}{\left(1 - \frac{1}{i}\right)\tau \leq t \leq \tau}} |F_{A_j} (x, t) | \leq i \quad (i \leq j)
\end{equation}
\begin{equation*}
U_i \Subset U_j  \quad (i < j).
\end{equation*}
The collection $\mathcal{S}$ is nonempty. For, we may let $x \in M \setminus \Sigma$ and $0 < R < R_0$ be such that a subsequence $\{j\}$ satisfies
\begin{equation*}
\Phi_j \left(R, x, \tau - R^2 \right)< \epsilon_0
\end{equation*}
for all $j.$ By the $\epsilon$-regularity Theorem 6.4 
of \cite{oliveirawaldron}, there exists $\delta > 0$ such that
\begin{equation}\label{uhlenbeck:deltabound}
\sup_{B_{\delta R}(x)\times \LB \left( 1 - \delta \right) \tau , \tau \RB} |F_{A_j} | \leq \frac{C}{\left( \delta R \right)^2}.
\end{equation}
We may then let $U_j = B_{(1 - 1/j) \delta R}(x)$ and choose $i = \lceil \frac{C}{(\delta R)^2} \rceil ,$ so that (\ref{uhlenbeck:deltabound}) implies (\ref{uhlenbeck:zorncondition}).

Define a partial ordering on $\mathcal{S}$ by
$$U_* \leq V_* \mbox{ if } \bigcup_i U_i \subset \bigcup_i V_i.$$
Letting
\begin{equation}\label{uhlenbeck:chain}
U^1_* \leq \cdots \leq U^k_* \leq U^{k + 1}_* \leq \cdots
\end{equation}
be a chain in $\mathcal{S},$ we may construct an upper bound $V_* \in \mathcal{S}$ by the following ``diagonal'' argument. We will construct an increasing sequence of elements $V^k_* \in \mathcal{S},$ 
and then let $V_* = \cup_{k} V^k_*.$

Assume without loss that $U^1_*$ is nonempty, and let $V^1_* = U^1_{\ell_1}$ for any $\ell_1 \in \mathcal{I}(U^1_*)$. Assuming that $V^k_*$ has been chosen, we may choose $\ell_{k + 1} \in \mathcal{I}(U^{k + 1}_*)$ with $\ell_{k + 1} > \ell_k,$ such that
\begin{equation}\label{uhlenbeck:exhaustinduction}
\bigcup_{\stackrel{m \leq k}{ i \leq \ell_k} } U^m_i \Subset U^{k + 1}_{\ell_{k + 1}}.
\end{equation}
We then let
$$V^{k + 1}_* = V^k_* \cup U^{k + 1}_{\ell_{k + 1}}.$$
By construction, the resulting set
$$V_* = \bigcup_k V^k_*$$
is an element of $\mathcal{S}.$ Since $\ell_k \to \infty$ as $k \to \infty,$ from (\ref{uhlenbeck:exhaustinduction}), we have
\begin{equation*}
 \bigcup_{m , i} U^m_i \subset \bigcup_{k,i} V^k_i = \bigcup_i V_i.
\end{equation*}
Therefore, $V_*$ is an upper bound in $\mathcal{S}$ for the given chain (\ref{uhlenbeck:chain}), as required.

We conclude from Zorn's lemma that $\mathcal{S}$ contains a maximal element, $W_*.$ Defining the singular set $\Sigma^W \supset \Sigma$ for the subsequence $\mathcal{I}(W_*)$ via (\ref{sigma:sigmadef}), we claim that $W_*$ is an exhaustion of $M \setminus \Sigma^W.$ For, if there existed $x \in M \setminus \left( \cup_i W_i \cup \Sigma^W \right),$ we could choose $R > 0$ and a further subsequence $J \subset \mathcal{I}(W_*)$ such that $\sup_{B_R(x), j \in J} |F_{A_j}| \leq i,$ for some $i.$ But then, defining $V_* \in \mathcal{S}$ by
\begin{equation}
V_j = W_j \cup B_{(1 - 1/j)R}(x)
\end{equation}
for all $j \in J$ with $j \geq i,$ we conclude that $W_*$ was not maximal, which is a contradiction. Therefore $\cup W_i = M \setminus \Sigma^W,$ as claimed.

Finally, we pass entirely to the subsequence $\mathcal{I}(W_*),$ which we relabel as $\{i \}_{i = 1}^\infty,$ and replace $\Sigma$ by $\Sigma^W.$ The assumption (\ref{uhlenbeck:zorncondition}) then becomes, for each $i \in \N$ and a certain $\tau_i < \tau,$ the crucial bound
\begin{equation}\label{uhlenbeck:curvbound}
\sup_{\stackrel{\stackrel{x \in U_i}{\tau_i \leq t < \tau }}{j \geq i} } \left| F_j(x, t) \right| < \infty.
\end{equation}
By Lemma \ref{lemma:antibubble}, (\ref{uhlenbeck:curvbound}) may be improved to the derivative estimates
\begin{equation}\label{uhlenbeck:derivests}
\sup_{\stackrel{\stackrel{x \in U_i}{ \tau_i \leq t < \tau }}{j \geq i} } \left| \nabla^{(k)}_j F_j(x, t) \right| < \infty
\end{equation}
for each $i,k \in \N.$

With (\ref{uhlenbeck:derivests}) now in hand, the construction of the bundle maps $u_i$ and Uhlenbeck limit $A_\infty$ follows the standard argument. By the Theorem of Uhlenbeck \cite{uhlenbecklp}, for each $x \in M \setminus \Sigma,$ there exists a ball $D \ni x$ and a gauge transformation $v_j$ on $D$ such that $v_j \left( A_j(\tau_j) \right)$ is in Coulomb gauge 
on $\left. E \right|_{D},$ for each $j \geq i.$ From (\ref{uhlenbeck:derivests}), the estimates of Donaldson-Kronheimer \cite{donkron}, Lemma 2.3.11, or Wehrheim \cite{wehrheimuhlenbeck}, Theorem 5.5, give uniform $C^k$ estimates on $v_j(A_j( \tau_j))$ over $D,$ for each $k \in \N$ (where the $C^k$ norms are defined with respect to a fixed reference connection $\nabla_{ref},$ see Remark \ref{rmk:referenceconn}). By the Arzela-Ascoli Theorem, there exists a subsequence which is smoothly convergent over $D.$ The sequence $A_i(\tau_i)$ therefore satisfies the property required by Corollaries \ref{cor:dkballs} and \ref{cor:exhaustion}, with 
$\Omega = M \setminus \Sigma;$ we obtain the required bundle $E_\infty \to M \setminus \Sigma,$ connection $A_\infty,$ and sequence of bundle maps $u_i,$ satisfying
\begin{equation}\label{uhlenbeck:uitauitoinfty}
u_i(A_i(\tau_i)) \to A_\infty \mbox{ as } i \to \infty
\end{equation}
in $C^\infty_{loc} \left( M \setminus \Sigma \right).$

We now turn to the proofs of (\ref{uhlenbeck:titotau}) and (\ref{uhlenbeck:uiaiconv}). Fix $k \in \N;$ for $i \geq k,$ we have
\begin{equation}\label{uhlenbeck:4}
\begin{split}
& \| u_i(A_i(t)) - A_\infty \|_{C^k(U_k)} \\
& \qquad \qquad \,\,\, \leq \| u_i(A_i(t)) - u_i(A_i(\tau_i)) \|_{C^k(U_k)} + \| u_i(A_i(\tau_i)) - A_\infty \|_{C^k(U_k)}.
\end{split}
\end{equation}
The second term on the RHS tends to zero with $i,$ by (\ref{uhlenbeck:uitauitoinfty}). Let
\begin{equation*}
\delta_i = \sqrt{\int_0^\tau \!\!\!\!\! \int_{U_{k+1}} |D^*F_i|^2 \, dV dt }
\end{equation*}
which, by the assumption (\ref{sigma:finitenergy}) and the local energy inequality, are uniformly bounded. 
By (\ref{uhlenbeck:curvbound}), for $i$ sufficiently large and $\tau_{k + 1} \leq t < \tau,$ we have 
\begin{equation}\label{uhlenbeck:Kcurvbound}
\| F_i(t) \|_{C^0(U_{k + 1})} \leq K
\end{equation}
for some $K > 0.$

To prove (\ref{uhlenbeck:titotau}), note that 
$u_i(A_i(t))$ are smooth solutions of (YM) on $\left. E_\infty \right|_{U_{k + 1}}.$ 
Applying Lemma \ref{lemma:distancecomp}, we obtain
\begin{equation}\label{uhlenbeck:5}
\begin{split}
\| u_i(A_i(t)) - u_i(A_i(\tau_i)) \|_{C^k(U_k)} & \leq C \delta_i \sqrt{| \tau_ i - t |} \left( 1 +  \| u_i(A_i(\tau_i)) \|^k_{C^{k -1} (U_k)} \right) \\
& \leq  C \delta_i \sqrt{| \tau_ i - t |}
\end{split}
\end{equation}
for $\tau_{k + 1} \leq t \leq \tau.$ We have absorbed the last factor because $u_i(A_i(\tau_i))$ is convergent, hence uniformly bounded in $C^{k -1} (U_k).$ Since the $\delta_i$ are bounded and $| \tau_i - t_i| \to 0,$ (\ref{uhlenbeck:5}) gives
$$\| u_i(A_i(t_i)) - u_i(A_i(\tau_i)) \|_{C^k(U_k)} \to 0.$$
Returning to (\ref{uhlenbeck:4}), we have
\begin{equation*}
\| u_i(A_i(t_i)) - A_\infty \|_{C^k(U_k)} \to 0 \,\,\, \mbox{ as } i \to \infty.
\end{equation*}
Since $k$ was arbitrary, this proves (\ref{uhlenbeck:titotau}).

To prove (\ref{uhlenbeck:uiaiconv}), we now assume (per (\ref{dstarftozero})) that
\begin{equation}\label{uhlenbeck:deltaito0}
\delta_i \to 0 \,\,\, \mbox{ as }i \to \infty.
\end{equation}
We claim that given any $\tau_0 > 0,$ for $i$ sufficiently large, a bound of the form (\ref{uhlenbeck:Kcurvbound}) will hold for all $0 < \tau_0 \leq t < \tau.$ This is easily seen from Lemma \ref{lemma:antibubble}, which may be applied on a cover of $U_{k + 1},$ in view of (\ref{sigma:e0bound}) and (\ref{uhlenbeck:deltaito0}). It also follows from (\ref{antibubble:derivests}) that $A_\infty$ is a Yang-Mills connection, since
\begin{equation*}
\begin{split}
\|D^*F_{u_i A_i(\tau_i)} \|_{C^0(U_k)} \leq C \|D^*F_{u_i A_i} \|_{L^2(M \times \LB 0, \tau \right) )} & = C \| D^*F_i \|_{L^2( M \times \LB 0, \tau \right))} \\
& \leq C \delta_i \to 0 \mbox{ as } i \to \infty
\end{split}
\end{equation*}
and $u_i A_i(\tau_i) \to A_\infty$ in $C^\infty(U_k).$

We now apply Lemma \ref{lemma:distancecomp}, to again obtain the bound (\ref{uhlenbeck:5}) for $0 < \tau_0 \leq t < \tau.$ Since $\delta_i (\tau_i - t) \to 0,$ we again obtain 
$$\| u_i(A_i(t)) - u_i(A_i(\tau_i)) \|_{C^k(U_k)} \to 0$$
and (\ref{uhlenbeck:4}) becomes
\begin{equation}\label{uhlenbeck:seeyalater}
\| u_i(A_i(t)) - A_\infty \|_{C^k(U_k)} \to 0 \mbox{ as } i \to \infty.
\end{equation}
Since $\tau_0$ and $k$ were arbitrary, this implies (\ref{uhlenbeck:uiaiconv}), completing the proof.
\end{proof}

\begin{proof}[Proof of Corollary \ref{cor:uhlenbeck}] Since $M$ is compact, we have
\begin{equation*}
\int_M |F_i (t)|^2 \, dV + 2 \int_0^t \!\!\!\! \int_M |D^*F_i (s) |^2 \, dV ds = \int_M |F_i (0)|^2 \, dV
\end{equation*}
for any $0 \leq t < \infty.$ Therefore, for any $\tau > 0$ and $t_i \nearrow \infty,$ we have
\begin{equation*}
\int_{t_i - \tau}^{t_i + \tau} \!\!\!\!\!\int_M |D^*F_i (t) |^2 \, dV dt \to 0.
\end{equation*}
We may therefore apply Theorem \ref{thm:uhlenbeck} to the sequence of solutions 
$$A_i(t) = A(t_i + t)$$
to obtain $A_\infty$ and bundle maps $u_i$ satisfying (\ref{uhlenbeck:uiaiconv}) over $\LB t_i + 1 , t_i + 2 \RB,$ say. By the argument just used in proving (\ref{uhlenbeck:seeyalater}) above, the limit extends over arbitrary time intervals. 
\end{proof}


\section{Proof of Theorem \ref{thm:linwangblowup}}\label{sec:blowup}

We first note the following straightforward variant of Theorems \ref{thm:sigma}-\ref{thm:uhlenbeck}: assume that $M_i \Subset M_\infty$ are open submanifolds which exhaust $M_\infty,$ and $g_i$ are metrics on $M_i$ such that
$$g_i \to g_\infty \mbox{ as } i \to \infty$$
in $C^\infty_{loc}(M_\infty, g_\infty).$
Then, letting $A_i$ be solutions of (YM) with respect to the metrics $g_i$ on bundles $E_i \to M_i$ of fixed structure group, we obtain a rectifiable singular set $\Sigma \subset M_\infty$ and Uhlenbeck limit $A_\infty$ on a bundle $E_\infty \to M_\infty \setminus \Sigma.$
Given a fixed manifold $(M,g)$ and $x_0 \in M,$ this version can be used for blowup analysis: let $(M_\infty, g_\infty) = (\R^n, g_{Euc})$ and
$$(M_i, g_i) = \left( B_{\rho_1(x_0)}(x_i), \lambda_i^{-2}g(x_i + \lambda_i x) \right)$$
for sequences $x_i \to x_0 \in M$ and  $\lambda_i \searrow 0.$

\begin{proof}[Proof of Theorem \ref{thm:linwangblowup}]

Let $\tau_i \nearrow \tau,$ and define the measures $\mu$ and $\nu$ by (\ref{sigma:nudef}), as in the proof of Theorem \ref{thm:sigma}.

Without loss of generality, we may take $x_0$ to be a point where the tangent measure $T_{x_0} \nu = T_{x_0} \mu$ is equal to a measure of constant density along an $(n-4)$-plane $V \subset T_{x_0} M.$ Let $\rho_i \searrow 0$ be a sequence such that the rescaled curvature measures converge
\begin{equation}\label{linwangblowup:tangentmeasuresconverge}
\rho_i^2 | F_i(x_0 + \rho_i x, \tau_i)| \to T_{x_0} \nu
\end{equation}
weakly on $\R^n.$ By Proposition \ref{prop:xiprop}, we may also assume that $x_0$ is such that
\begin{equation}\label{linwangblowup:originaldstarflimit}
\lim_{\sigma \nearrow \tau} \limsup_{R \searrow 0} \underline{\Xi}(R, x_0, \sigma, \tau) = 0.
\end{equation}

Before proceeding with the proof, we replace the original sequence $A_i$ with the blown-up sequence of solutions $\rho_i A_i(x_0 + \rho_i x, \tau_i + \rho_i^2 t),$ which solve (YM) on an exhaustion of $T_{x_0} M \times \left( -\infty, 0 \right)$ with respect to the rescaled metrics $\rho_i^{-2} g(x_0 + \rho_i x)$ (converging locally uniformly to the Euclidean metric).

Notice that the assumption (\ref{sigma:finitenergy}) is preserved by parabolic rescaling, due to the monotonicity formula. Applying the variant of Theorem \ref{thm:sigma} discussed above to the rescaled sequence, we have $\Sigma = V,$ and (\ref{linwangblowup:originaldstarflimit}) becomes
\begin{equation}\label{linwangblowup:rescaleddstarflimit}
\liminf_{i \to \infty} \int_\sigma^{0} \!\!\!\!\! \int_M |D^*F_i(x,t)|^2 u_{0,R}(x) \varphi^2_{0, \rho_1(x_0)/ \rho_i}(x) \, dV_x dt = 0
\end{equation}
for any $R>0$ and $-\infty < \sigma < 0.$ We may pass to a subsequence such that for $R = 1,$ $\liminf$ may be replaced by $\lim$ in (\ref{linwangblowup:rescaleddstarflimit}), 
so that (\ref{dstarftozero}) is satisfied. By Theorem \ref{thm:uhlenbeck}, the assumption (\ref{linwangblowup:tangentmeasuresconverge}) implies that $F_{A_\infty} \equiv 0,$ so we have
\begin{equation}\label{linwangblowup:FAitozeropointwise}
\left| F_{A_i}(x,t) \right| \to 0 \mbox{ as } i \to \infty
\end{equation}
locally uniformly on $\left( T_{x_0}M \setminus V \right) \times \left( - \infty, 0 \right).$ Also, by Lemma \ref{lemma:antibubble} and (\ref{linwangblowup:rescaleddstarflimit}), we have
\begin{equation}\label{linwangblowup:phigeq0onV}
\liminf_{R \to 0}\underline{\Phi}(R, x, t) \geq \epsilon_0
\end{equation}
for all $x \in V$ and $t < 0.$

For $x \in T_{x_0}M,$ write
$$x = (y, z)$$
where $y \in V\cong \R^{n-4}$ and $z \in V^{\perp} \cong \R^4,$ and choose coordinates such that $V$ is spanned by $e_1, \ldots, e_{n-4}.$ Define
\begin{equation}
\begin{split}
f_i(y) & = \int_{-1}^{0} \!\! \int_{\{y\} \times B_1^{4}(0)} |D^*F_i(y,z,t)|^2 \, dV_z dt \\
g_i(y,t) & = \int_{\{y\} \times B_1^{4}(0)} \sum_{\alpha = 1}^{n-4} \sum_{\beta = 1}^{n} | \left(F_i\right)_{\alpha \beta} (y,z,t) |^2 \, dV_z \\
h_i(y,t) & = \int_{\{y\} \times B_1^{4}(0)} |F_i(y,z,t)|^2 \, dV_z.
\end{split}
\end{equation}
By (\ref{linwangblowup:rescaleddstarflimit}), for any $y_0 \in V,$ we have
\begin{equation}\label{linwangblowup:fiinttozero}
\int_{B_1^{n-4}(y_0)} f_i (y) \, dV_y \to 0 \mbox{ as } i \to \infty.
\end{equation}
Given (\ref{linwangblowup:fiinttozero}), the monotonicity-formula trick of Lin-Wang \cite{linwangbook}, p. 211, implies 
\begin{equation}\label{linwangblowup:giinttozero}
\int_{-1}^{-1/8} \!\!\!\! \int_{B_1^{n-4}(0)} g_i (y,t) \, dV_y dt \to 0 \mbox{ as } i \to \infty.
\end{equation}
The monotonicity formula also implies
\begin{equation}\label{linwangblowup:hibounded}
\int_{B_1^{n-4}(0) } h_i(y,t) \, dV_y \leq C E
\end{equation}
for all $-1 \leq t < 0.$

Define the Hardy-Littlewood maximal functions
\begin{equation}
\begin{split}
M(f_i)(y) & = \sup_{0 < R \leq 1} R^{4-n} \int_{B_R^{n - 4}(0)} f_i(w) \, dV_w \\
M(g_i)(y, t) & = \sup_{0 < R \leq 1} R^{2 - n} \int_{t - R^2}^{t} \!\! \int_{B_R^{n - 4}(0)} g_i(w,s) \, dV_w ds \\
M(h_i)(y, t) & = \sup_{0 < R \leq 1} R^{4 - n} \!\! \int_{B_R^{n - 4}(0)} h_i(w,t) \, dV_w.
\end{split}
\end{equation}
Letting $L_i = \int_{B_1^{n-4}(y_0)} f_i (y) \, dV_y,$ by the weak $L^1$-estimate for the maximal function, we have
\begin{equation}
\mu_{Leb} \{ M(f_i) \geq \epsilon \} \leq C \frac{L_i} {\epsilon}
\end{equation}
for any $\epsilon > 0.$ Since $L_i \to 0$ (\ref{linwangblowup:fiinttozero}), we may pass to a subsequence for which there exist points $y_i \in B_{1/2}^{n-4}(0)$ such that
\begin{equation}\label{linwangblowup:Mfitozero}
M(f_i)(y_i) \to 0.
\end{equation}
Arguing similarly, by (\ref{linwangblowup:giinttozero}), we may choose $t_i \in \LB -1/2, -1/4 \RB$ such that
\begin{equation}\label{linwangblowup:Mgitozero}
M(g_i)(y_i, t_i) \to 0.
\end{equation}
Also, by (\ref{linwangblowup:hibounded}), we may assume that $y_i, t_i$ are such that
\begin{equation}\label{linwangblowup:finiteenergymaximal}
M(h_i)(y_i, t_i - \delta_i^2) \leq C E.
\end{equation}

For $y, s \in \R^{n-4}$ and $R > 0,$ let
\begin{equation}\label{linwangblowup:vdef}
v_{R, y}(s) = \frac{R^{4-n}}{(4\pi)^{n/2}} \exp \left( - \frac{\left| y - s \right|^2}{4R^2} \right).
\end{equation}
Then, since
$$R^{4-n} \exp \left( - \frac{r^2}{4R^2} \right) \leq C\left( r + R \right)^{4-n} \exp \left( - \frac{r^2}{2R^2} \right)$$
it is easily seen from (\ref{linwangblowup:Mfitozero}-\ref{linwangblowup:Mgitozero}) that for any $y \in \R^{n-4},$ we have
\begin{equation}\label{linwangblowup:unrescaledfigitozero}
\begin{split}
& \sup_{0 < R \leq 1} \int_{B_1^{n - 4}(0)} f_i(s) v_{R, y_i + Ry}(s) \, dV_s \to 0 \\
& \sup_{0 < R \leq 1} R^{-2} \int_{t_i - R^2}^{t_i} \!\! \int_{B_1^{n - 4}(0)} g_i(s) v_{R, y_i + Ry}(s) \, dV_s dt \to 0 
\end{split}
\quad \mbox{ as } i \to \infty
\end{equation}
locally uniformly in $y.$
Since $|F_i|$ are continuous, in light of (\ref{linwangblowup:FAitozeropointwise}), for $i$ sufficiently large it is possible to choose $\delta_i > 0$ and $z_i \in B_{1/4}^{4}(0)$ such that
\begin{equation}\label{linwangblowup:Phimaxepsilon1}
\begin{split}
\Phi_{0,1}\left(A_i; \delta_i, (y_i,z_i), t_i \right) & = \int \left| F_i (x,t_i) \right|^2 u_{(y_i, z_i), \delta_i}(x) \varphi^2_{0,1}(x) \, dV_x \\
& = \frac{\epsilon_0}{ 2} = \max_{z \in B^4_{1/2}} \Phi_{0,1}\left(A_i; \delta_i, (y_i,z), t_i \right).
\end{split}
\end{equation}
It follows from (\ref{linwangblowup:FAitozeropointwise}) and (\ref{linwangblowup:phigeq0onV}) that $\delta_i \to 0$ and $z_i \to 0 \in \R^4$ as $i \to \infty.$

Finally, we rescale further to define the sequence
\begin{equation}
\bar{A}_i \left( (y,z), t \right) = \delta_i A_i \left( (y_i + \delta_i y, z_i + \delta_i z), t_i + \delta_i^2 (t + 1) \right).
\end{equation}
Passing again to a subsequence, let $B(x)$ be the Uhlenbeck limit of $\bar{A}_i$ on $\R^n \cong T_x M,$ with $\tau = 0,$ per Theorems \ref{thm:sigma}-\ref{thm:uhlenbeck}. For the rescaled sequence, (\ref{linwangblowup:unrescaledfigitozero}) with $R = \delta_i$ becomes
\begin{align}\label{linwangblowup:rescaledfitozero}
\xi^2_i(x) \, & \hat{=} \int_{-c \delta_i^{-2}}^{c \delta_i^{-2}} \!\!\! \int_{B_{\delta_i^{-1}}} \left| D^*F_{\bar{A}_i} \right|^2 u_{1, x} \, dV dt \to 0 \\
\label{linwangblowup:rescaledgitozero} \eta^2_i(x) \, & \hat{=} \int_{-1}^{0} \!\! \int_{B_{\delta_i^{-1}}} \sum_{\alpha = 1}^{n-4} \sum_{\beta = 1}^{n} \left| \left(F_{\bar{A}_i}\right)_{\alpha \beta} \right|^2 u_{1, x} \, dV dt \to 0
\end{align}
as $i \to \infty,$ for any $x \in \R^n.$ We have replaced $v_{1, y},$ defined by (\ref{linwangblowup:vdef}), by $u_{1, x} = u_{1, (y,z)} \leq v_{1, y},$ defined by (\ref{udef}). The bound (\ref{linwangblowup:finiteenergymaximal}) becomes
\begin{equation}\label{linwangblowup:finiteenergy}
\int_{B_1^{n-4}(0) \times B_{1/ \delta_i}^{4}(0)} |F_{\bar{A}_i}(x,- 1)|^2 \, dV_x \leq C E.
\end{equation}
The statement (\ref{linwangblowup:Phimaxepsilon1}) becomes
\begin{equation}\label{linwangblowup:rescaledPhimaxepsilon1}
\Phi_i\left(1, 0, -1 \right) = \frac{\epsilon_0}{2}  = \max_{z \in \R^4} \Phi_i \left(1, (0,z), -1 \right).
\end{equation}
The cutoff function $\varphi_{0,1}$ of (\ref{linwangblowup:Phimaxepsilon1}) is negligible after the rescaling, and we will suppress it in the following calculations.

We shall argue that the convergence $\bar{A_i} \to B$ is smooth over all compact subsets of $\R^n,$ and $B(x)$ splits as a product. To this end, 
we compute
\begin{equation}\label{linwangblowup:partialcalc}
\begin{split}
\frac{\del}{\del x^\alpha} |F|^2 & = \LA \nabla_\alpha F_{\beta\gamma}, F_{\beta\gamma} \RA \\
& = \LA \left( \nabla_\beta F_{\alpha \gamma} + \nabla_\gamma F_{\beta \alpha} \right), F_{\beta\gamma} \RA \\
& = 2 \LA \nabla_\beta F_{\alpha \gamma}, F_{\beta \gamma} \RA \\
& = 2 \left( \nabla_\beta \LA F_{\alpha \gamma}, F_{\beta \gamma} \RA - \LA F_{\alpha \gamma}, \nabla_\beta F_{\beta \gamma} \RA \right) \\
& = 2 \left( \nabla_\beta \LA F_{\alpha \gamma}, F_{\beta \gamma} \RA + \LA F_{\alpha \gamma}, D^*F_\gamma \RA \right).
\end{split}
\end{equation}
Then
\begin{equation}\label{linwangblowup:delxalphaphii}
\begin{split}
\frac{\del}{\del x^\alpha} \Phi_i(1, x, t) & = \int |F_i(w,t)|^2 \frac{\del}{\del x^\alpha}u_{1, x}(w) \, dV_w \\
& = \int \frac{\del}{\del w^\alpha} |F_i(w,t)|^2 u_{1, x}(w) \, dV_w \\
& = 2\int \left( - \LA F_{\alpha \gamma}, F_{\beta \gamma} \RA \frac{\del}{\del w^\beta} u_{1, x}(w) + \LA F_{\alpha \gamma}, D^*F_\gamma \RA u_{1, x}(w) \right)  \, dV_w
\end{split}
\end{equation}
where we have used (\ref{linwangblowup:partialcalc}) in the third line. Define
$$\Psi_i(R, x) = \int_{-1}^0 \Phi_i(R,x,t) \, dt.$$
By (\ref{uxr2calc}), we have
$$|\nabla u_{1, x}(w)| \leq C \sqrt{ u_{1, x}u_{2, x}(w)}.$$
Letting $\alpha \in \{ 1, \ldots, n-4 \},$ we may integrate (\ref{linwangblowup:delxalphaphii}) in time and apply H\"older's inequality, to obtain
\begin{equation}\label{linwangblowup:deltalphaphiinext}
\begin{split}
\left| \frac{\del}{\del y^\alpha} \Psi_i(1, x) \right| \leq C \eta_i(x) \left( \Psi_i(2, x) + \xi_i(x) \right).
\end{split}
\end{equation}
By (\ref{linwangblowup:rescaledfitozero}-\ref{linwangblowup:rescaledgitozero}), 
the RHS of (\ref{linwangblowup:deltalphaphiinext}) tends to zero locally uniformly in $x$ as $i \to \infty.$ Then for any $y \in \R^{n-4}$ and $z \in \R^4,$ (\ref{linwangblowup:deltalphaphiinext}) implies
\begin{equation*}
\left| \Psi_i(1, (y, z)) - \Psi_i(1, (0, z))  \right|  \to 0 \mbox{ as } i \to \infty.
\end{equation*}
From (\ref{linwangblowup:rescaledPhimaxepsilon1}), we have
\begin{equation}\label{linwangblowup:Psii0andy}
\limsup_{i \to \infty} \Psi_i(1, (y, z)) = \limsup_{i \to \infty} \Psi_i(1, (0, z)) \leq \frac{\epsilon_0}{2}.
\end{equation}
But from Lemma \ref{lemma:antibubble} and (\ref{linwangblowup:rescaledfitozero}), we have
$$\limsup_{i \to \infty} \Psi_i(1, x) = \limsup_{i \to \infty} \Phi_i(1, x, -1) = \limsup_{i \to \infty}\Phi_i(1, x, t)$$
for any $t \in \R.$ Therefore (\ref{linwangblowup:Psii0andy}) yields
\begin{equation}
\limsup_{i \to \infty} \Phi_i(1, x, t) \leq \frac{\epsilon_0}{2}
\end{equation}
for all $x \in \R^n$ and $t \in \R.$ Therefore $\Sigma = \emptyset,$ and the limit $\bar{A}_i \to B$ takes place smoothly on $\R^{n + 1}.$

Since the convergence is strong on compact sets and the energy is locally bounded, 
(\ref{linwangblowup:finiteenergy}) and (\ref{linwangblowup:rescaledPhimaxepsilon1}) are preserved in the limit. Therefore $B$ is not flat, and has finite energy on the strip $B_1^{n-4}(0) \times \R^4.$ By (\ref{linwangblowup:rescaledgitozero}), $V \intprod F_B = 0,$ hence $B(y,z)$ 
reduces to a finite-energy Yang-Mills connection on $\R^4,$ as claimed.
\end{proof}

\bibliographystyle{amsinitial}
\bibliography{biblio}

\end{document}